\documentclass[12pt, reqno]{amsart}

\usepackage{amssymb, amsmath, amsthm, mathrsfs}
\usepackage[backref, colorlinks=true]{hyperref}
\usepackage[alphabetic,backrefs,lite]{amsrefs}
\usepackage{amscd}   
\usepackage{fullpage}
\usepackage[all]{xy} 
\usepackage[OT2, T1]{fontenc}

\DeclareFontEncoding{OT2}{}{} 

\usepackage[usenames,dvipsnames]{color}

\newcommand{\revedit}[1]{{\color{red} \sf #1}}

\newcommand{\hideqed}{\renewcommand{\qed}{}} 

\newtheorem{lemma}{Lemma}[section]
\newtheorem{theorem}[lemma]{Theorem}
\newtheorem{proposition}[lemma]{Proposition}
\newtheorem{prop}[lemma]{Proposition}
\newtheorem{cor}[lemma]{Corollary}
\newtheorem{conj}[lemma]{Conjecture}

\newtheorem{claim*}{Claim}
\newtheorem{thm}[lemma]{Theorem}
\newtheorem{question}[lemma]{Question}
\newtheorem{defn}[lemma]{Definition}

\theoremstyle{remark}
\newtheorem{remark}[lemma]{Remark}
\newtheorem{remarks}[lemma]{Remarks}

\newcommand{\A}{{\mathbb A}}
\newcommand{\G}{{\mathbb G}}

\newcommand{\F}{{\mathbb F}}
\newcommand{\Q}{{\mathbb Q}}
\newcommand{\R}{{\mathbb R}}
\newcommand{\Z}{{\mathbb Z}}

\newcommand{\Xbar}{{\overline{X}}}
\newcommand{\Qbar}{{\overline{\Q}}}

\newcommand{\kbar}{{\overline{k}}}

\newcommand{\Ybar}{{\overline{Y}}}
\newcommand{\Abar}{{\overline{A}}}
\newcommand{\Ebar}{{\overline{E}}}

\newcommand{\Etilde}{{\widetilde{E}}}

\newcommand{\calA}{{\mathcal A}}

\newcommand{\calH}{{\mathcal H}}

\newcommand{\calM}{{\mathcal M}}

\newcommand{\OO}{{\mathcal O}}

\newcommand{\scrO}{{\mathscr O}}

\newcommand{\epsbar}{\overline{\varepsilon}}
\newcommand{\Xns}{X^{t,\epsbar}_{ns}}
\newcommand{\XB}{X^t_{B}}

\DeclareMathOperator{\Alb}{Alb}

\DeclareMathOperator{\HH}{H}
\DeclareMathOperator{\tr}{tr}
\DeclareMathOperator{\Tr}{Tr}

\DeclareMathOperator{\coker}{coker}

\DeclareMathOperator{\im}{im}

\DeclareMathOperator{\End}{End}

\DeclareMathOperator{\Hom}{Hom}

\DeclareMathOperator{\Aut}{Aut}
\DeclareMathOperator{\Gal}{Gal}

\DeclareMathOperator{\Res}{Res}

\DeclareMathOperator{\Br}{Br}

\DeclareMathOperator{\Sym}{Sym}

\DeclareMathOperator{\Pic}{Pic}

\DeclareMathOperator{\SPEC}{\bf Spec}

\DeclareMathOperator{\tors}{tors}
\DeclareMathOperator{\et}{et}

\DeclareMathOperator{\rank}{rank}
\DeclareMathOperator{\Mat}{M}
\DeclareMathOperator{\Kum}{Kum}
\DeclareMathOperator{\NS}{NS}

\DeclareMathOperator{\SL}{SL}
\DeclareMathOperator{\GL}{GL}

\DeclareMathOperator{\M}{M}

\DeclareMathOperator{\disc}{disc}

\DeclareMathOperator{\Bl}{Bl}
\DeclareMathOperator{\per}{per}

\newcommand{\isom}{\cong}

\newcommand{\into}{\hookrightarrow}
\newcommand{\onto}{\twoheadrightarrow}
\newcommand{\To}{\longrightarrow}

\numberwithin{equation}{section}
\numberwithin{table}{section}

\newcommand{\defi}[1]{\textsf{#1}} 

\title{Abelian $n$-division fields of elliptic curves and \\Brauer groups of  product Kummer \& abelian surfaces}
\author{Anthony V\'arilly-Alvarado}

\address{Department of Mathematics MS 136, Rice University, Houston, TX
			77005, USA}
\email{varilly@rice.edu}
\urladdr{http://math.rice.edu/\~{}av15}

\author{Bianca Viray}

\address{University of Washington, Department of Mathematics, Box 354350, Seattle, WA 98195, USA}
\email{bviray@math.washington.edu}
\urladdr{http://math.washington.edu/\~{}bviray}

\date{}

\begin{document}

    \maketitle
    
    \begin{abstract}
    	Let $Y$ be a principal homogeneous space of an abelian surface, or a K3 surface, over a finitely generated extension of $\Q$. In 2008, Skorobogatov and Zarhin showed that the Brauer group modulo algebraic classes $\Br Y/\Br_1 Y$ is finite.  We study this quotient for the family of surfaces that are geometrically isomorphic to a product of isogenous non-CM elliptic curves, as well as the related family of geometrically Kummer surfaces; both families can be characterized by their geometric N\'eron-Severi lattices.
		 Over a field of characteristic $0$, we prove that the existence of a strong uniform bound on the size of the odd-torsion of $\Br Y/\Br_1 Y$  is equivalent to the existence of a strong uniform bound on integers $n$ for which there exist non-CM elliptic curves with abelian $n$-division fields.  Using the same methods we show that, for a fixed prime $\ell$, a number field $k$ of fixed degree $r$, and a fixed discriminant of the geometric N\'eron-Severi lattice,  $\#(\Br Y/\Br_1 Y)[\ell^\infty]$ is bounded by a constant that depends only on $\ell$, $r$, and the discriminant.
	\end{abstract}

\section{Introduction}\label{s:intro}%
    
	Let $k$ be a field of characteristic $0$, and let $\kbar/k$ be a fixed algebraic closure. Let $Y$ be a smooth projective surface over $k$ with trivial canonical sheaf $\omega_Y \isom \scrO_Y$, and let $\Ybar := Y\times_k\kbar$.  The Enriques-Kodaira classification of smooth algebraic surfaces shows that $\Ybar$ is either a K3 surface, in which case $\textup{h}^1(Y,\scrO_Y) = 0$, or an abelian surface, in which case $\textup{h}^1(Y, \scrO_Y) = 2.$ 
	
	If $k$ is algebraically closed, the Brauer group $\Br Y := \HH^2_{\et}(Y,\G_m)$ of $Y$  is isomorphic to $(\Q/\Z)^{b_2 - r}$, where $b_2$ is the second Betti number of $Y$, and $r$ denotes the rank of the N\'eron-Severi group of $Y$.  In particular, $\Br Y$ is infinite, because $k$ has characteristic $0$ and $\omega_Y\isom \scrO_Y$ so we have $b_2 - r \geq 2$. In stark contrast, the group $\im\left(\Br Y \to \Br \Ybar\right)$ is finite whenever $k$ is a finitely generated extension of $\Q$; this is a remarkable result of Skorobogatov and Zarhin~\cite{SZ-Finiteness}.  We are interested in the existence of bounds for the size of $\im\left(\Br Y \to \Br \Ybar\right)$ as $Y$ varies in a family with fixed geometric properties.  
	
	\begin{question}\label{question}
		Let $k$ be a field that is finitely generated over $\Q$. Let $Y$ be a smooth projective surface over $k$ with trivial canonical sheaf.  Is there a bound for 
		\[
			\#\im\left(\Br Y \to \Br \Ybar\right)
		\]
that is independent of $Y$, depending only on, say, $h^1(Y,\scrO_Y)$, the geometric N\'eron-Severi lattice $\NS\Ybar$ (considered as an abstract lattice), and $k$ or, if $k/\Q$ is finite, $[k:\Q]$?
	\end{question} 
    \begin{remarks}\hfill
		\begin{enumerate}
		\item[(i)] When $k$ is a number field and $Y$ is a K3 surface, the existence of bounds for $\#\im\left(\Br Y \to \Br \Ybar\right)$ that can depend on $Y$ has been studied previously; see~\cite{SZ-Kummer} and~\cite{Newton-BrauerK3} for the case of product Kummer surfaces, \cite{HKT-Degree2K3} for the case of K3 surfaces of degree $2$, and~\cite{CFTTV} for Kummer surfaces associated to Jacobians of genus $2$ curves.

		\item[(ii)] The Brauer group has a natural filtration 
		\[
			\Br_0 Y := \im(\Br k \to \Br Y) \subseteq \Br_1 Y := \ker(\Br Y \to \Br \Ybar) \subseteq \Br Y,
		\]
giving an injection $\Br Y/\Br_1 Y \hookrightarrow \Br \Ybar$, so we may replace $\#\im\left(\Br Y \to \Br \Ybar\right)$ in Question~\ref{question} with $\#(\Br Y/\Br_1 Y)$.  If $Y$ is a K3 surface, $\#(\Br_1 Y/\Br_0 Y)$ is bounded above by a constant that depends only on the rank of $\NS \Ybar$ (Lemma~\ref{lem:AlgBrBound}).  Therefore, bounding $\#(\Br Y/\Br_1 Y)$ is equivalent to bounding $\#(\Br Y/\Br_0 Y)$.  In contrast, if $Y$ is a principal homogeneous space of an abelian surface, then the quotient $\Br_1 Y/\Br_0 Y$ is infinite.
		\end{enumerate}
	\end{remarks}

	We study Question~\ref{question} for the following particular families of smooth projective surfaces:
	\begin{align*}
		\mathscr{A}_d^N := & \left\{Y/k : \omega_Y \isom \OO_Y, \textup{h}^1(Y, \OO_Y) = 2, \NS \Ybar \isom \begin{pmatrix} 0 & 1 & 1\\1& 0 & d\\1 & d & 0\end{pmatrix}, \Alb^1(Y) \in \HH^1(k,\Alb(Y))_N \right\},\\
		\mathscr{K}_d := & \left\{X/k : \omega_X \isom \OO_X, \textup{h}^1(X, \OO_X) = 0, \NS \Xbar \isom \Lambda_d\right\},
	\end{align*}
	where $\Lambda_d$ denotes the N\'eron-Severi lattice of a Kummer surface of a product of non-CM elliptic curves that have a cyclic isogeny of degree $d$ between them.  (The lattice $\Lambda_d$ is independent of the choice of elliptic curves and the cyclic isogeny of degree $d$, see~\S\ref{sec:KummerProducts}.) Note that the surfaces in $\mathscr{A}_d^N$ are geometrically abelian (so in particular, $Y\isom \Alb^1(Y)$) and the surfaces in $\mathscr{K}_d$ are K3.  In addition, if $d\neq d'$, then $\mathscr{K}_d\cap\mathscr{K}_{d'} = \mathscr{A}_d^N \cap\mathscr{A}_{d'}^{N'} = \emptyset$ for any integers $N$ and $N'$.  We prove the following theorem.

    \begin{thm}\label{thm:EquivalentUniformBoundsArbField}
    Let $d$ be a positive integer and let $F$ be a field of characteristic $0$.  For any positive integer $n$, we let $n_{\textup{odd}}$ denote the maximal odd integer that divides $n$. The following uniform boundedness statements are equivalent.
    \begin{enumerate}
    	\item[\textbf{(K3)}] For all positive integers $r$, there exists a $B = B(r,d)$ such that for all fields $k/F$ with $[k:F] \leq r$ and all surfaces $X/k\in \mathscr{K}_d$, 
    	\[
    		\#\left(\frac{\Br X}{\Br_0 X}\right)_{\textup{odd}} \leq B.
    	\]
		
    	\item[\textbf{(Ab)}] For all positive integers $r'$, there exists a $B' = B'(r,d)$ such that for all fields $k'/F$ with $[k':F] \leq r'$ and all surfaces $Y/k \in \mathscr{A}_d^2$, 
    	\[
    		\#\left(\frac{\Br Y}{\Br_1 Y}\right)_{\textup{odd}} \leq B'.
    	\]

    	\item[\textbf{(EC)}] For all positive integers $r''$, there exists a $B'' = B''(r'', d)$ such that for all fields $k''/F$ with $[k'':F] \leq r''$ and all non-CM elliptic curves $E/k''$ with a $k''$-rational cyclic subgroup of order $d$,
    	\[
    		\Gal(k''(E_n)/k'') \textup{ is abelian} \Rightarrow n_{\textup{odd}} \leq B''.
    	\]
    \end{enumerate}
    \end{thm}
    \begin{remark}
        The equivalence of \textbf{(Ab)} and \textbf{(K3)} primarily follows from a result of Skorobogatov and Zarhin~\cite[Thm. 2.4 and its proof]{SZ-Kummer}.  
    \end{remark}
    
    Theorem~\ref{thm:EquivalentUniformBoundsArbField} reduces Question~\ref{question} for the odd part of the Brauer group in the case of $\mathscr{K}_d$ and $\mathscr{A}_d^2$ to a question about the existence of abelian $n$-division fields for $n$ arbitrarily large.  If we fix a prime $\ell$ and specialize to the case of a number field, then this reduction, together with results of Abramovich~\cite{Abramovich-Gonality} and Frey~\cite{Frey-InfinitelyManyDegreed}, gives the existence of an unconditional uniform bound on the $\ell$-primary part of $\#(\Br Y/\Br_1 Y)$ and $\#(\Br X/\Br_0 X)$.
     
    \begin{theorem}\label{thm:ellprimaryAb}
		Fix a prime integer $\ell$ and positive integers $d, r$, and $v$. There is a positive constant $B := B(\ell,d,r,v)$ such that for all degree $r$ number fields $k$, all positive integers $N$ with $v_{\ell}(N)\leq v$, and all surfaces $Y/k\in \calA_d^N$,  we have
		\[
			\#(\Br Y/\Br_1 Y)[\ell^\infty] < B.
		\]
    \end{theorem}
    
	\begin{cor}\label{cor:ellprimaryK3}
		Fix a prime integer $\ell$ and positive integers $d$ and $r$. There is a positive constant $B := B(\ell,d,r)$ such that for all degree $r$ number fields $k$ and surfaces $X/k\in \mathscr{K}_d$, we have
		\[
			\#(\Br X/\Br_0 X)[\ell^\infty] < B.
		\]
	\end{cor}
    \begin{remark}
        Theorem~\ref{thm:ellprimaryAb} and Corollary~\ref{cor:ellprimaryK3} can be used to prove a slightly stronger version of Theorem~\ref{thm:EquivalentUniformBoundsArbField} in the case where $F$ is a number field.  See Theorem~\ref{thm:EquivalentUniformBoundsNumberField}.
    \end{remark}
    If we specialize our family further, then we may use work of Gonz\'alez-Jim\'enez and Lozano-Robledo~\cite{GJLR-AbelianDivisionFields} to obtain precise bounds.
    \begin{theorem}\label{thm:overQ}
        Let $d$ be a positive integer, let $E/\Q$ be a non-CM elliptic curve with a $\Q$-rational cyclic subgroup $C$ of order $d$, let $Y = E \times E/C$, and let $X := \Kum(Y).$  Then
        \[
            \#(\Br Y/\Br_1 Y),\; \#(\Br X/\Br_1 X) \leq (8d)^3.
        \]
    \end{theorem}
    \begin{remark}
        Primarily by work of Mazur~\cite{Mazur-Isogenies}, if $E/\Q$ is an elliptic curve with a $\Q$-rational subgroup of degree $d$, then $d\leq 163$.  Thus, for any $d$ and any $Y$ and $X$ as above, $\#(\Br Y/\Br_1 Y)$ and $\#(\Br X/\Br_1 X)$ are bounded above by $(8\cdot 163)^3$.
    \end{remark}
    
    \subsection{Heuristics toward a uniform bound}

        The above results might lead one to speculate that Question~\ref{question} has a positive answer in the case of number fields, at least for some N\'eron-Severi lattices.  We provide some heuristic arguments in this direction.

        \subsubsection{The case of surfaces associated to products of isogenous
non-CM elliptic curves}

            Let $E$ be a non-CM elliptic curve over a number field $k$.  Serre's Open Image Theorem states that the image of the Galois representation $\rho_E$ attached to $E$ is open in $\GL_2(\hat{\Z})$~\cite{Serre-OpenImage}.  In particular, there are finitely many integers $n$ such that $k(E_n)/k$ is an abelian extension.  In the same paper, Serre asked whether a \emph{uniform} version of his Open Image Theorem is true, i.e., if there exists a bound for the index $[\GL_2(\hat{\Z}):\im \rho_E]$ that depends only on $k$.  (Work of Mazur~\cite{Mazur-Isogenies} and Bilu-Parent-Rebolledo~\cite{BPR-p^r} strongly supports a positive answer, at least in the case $k=\Q$.) If there is such a bound and if, in addition, the bound can be taken to depend only on $[k:\Q]$ rather than $k$, then Theorem~\ref{thm:EquivalentUniformBoundsNumberField} proves {Question~\ref{question}} has a positive answer for $\mathscr{K}_d$ and $\mathscr{A}^2_d$.  Under the same assumptions, a similar argument proves {Question~\ref{question}} has a positive answer for $\mathscr{A}^N_d$ for any $N$.
			
        \subsubsection{Cohomological interpretation of torsion on elliptic curves}
	    
            Let $E$ be an elliptic curve over a number field $k$.  The torsion subgroup $E(k)_{\tors}$ of the $k$-points is a finite abelian group.  In~\cite{Manin}, Manin showed that, for a fixed prime $\ell$, the cardinality of the group $E(k)[\ell^\infty]$ of $\ell$-primary order points of $E$ is bounded in terms of $\ell$, independent of $E$.  Soon after, Mazur showed that, over $\Q$, the bound could be taken independent of $\ell$ and, more precisely, that there are only 15 possibilities for $E(\Q)_{\tors}$~\cite{Mazur}.  Following further progress by Kamienny~\cite{Kamienny}, Merel proved the strong uniform boundedness conjecture for elliptic curves: given a positive integer $d$, there is a constant $c := c(d)$ such that $\#E(k)_{\tors} < c(d)$ for all elliptic curves $E$ over any number field $k$ of degree $d$~\cite{Merel}.
	
		Among smooth projective curves, elliptic curves are precisely those that have trivial canonical sheaf.  Surfaces with this property fall into two geometric classes: K3 surfaces and abelian surfaces.  Hence, one might wonder if there is an analogous statement of Merel's theorem that holds for \emph{both} of these classes of surfaces.  The chain of group isomorphisms,
        	\begin{align*}
        		E(k)_{\tors} 	&\isom (\Pic^0 E)_{\tors} \isom (\Pic E)_{\tors} \\
        						&\isom \HH^1_{\textup{Zar}}(E,\scrO_E^\times)_{\tors} \isom \HH^1_{\et}(E,\G_m)_{\tors} \\
        						&\isom \im\left(\HH^1_{\et}(E,\G_m)_{\tors} \to \HH^1_{\et}(\Ebar,\G_m)_{\tors}\right) \\
                                	&\isom \im\left(\HH^{\dim E}_{\et}(E,\G_m)_{\tors} \to \HH^{\dim E}_{\et}(\Ebar,\G_m)_{\tors}\right),
        	\end{align*}
suggests that the group $\im\left(\HH^2_{\et}(Y,\G_m)_{\tors} \to \HH^2_{\et}(\Ybar,\G_m)_{\tors}\right)$ is a plausible replacement for $E(k)_{\tors}$. This group is precisely $\im\left(\Br Y \to \Br \Ybar\right)$, and Skorobogatov and Zarhin's finiteness result for it is a suitable replacement for the finiteness of $E(k)_{\tors}$.  In this light, Question~\ref{question} asks whether the analogous statement of Merel's theorem holds for K3 surfaces and principal homogeneous spaces of abelian surfaces, after possibly fixing a N\'eron-Severi lattice; Theorem~\ref{thm:ellprimaryAb} and Corollary~\ref{cor:ellprimaryK3} are analogues of Manin's result for particular lattices.

        \subsubsection{Moduli of K3 surfaces with level structure}\label{ss:K3moduli}
			
			Uniform boundedness statements for torsion on elliptic curves are equivalent to statements about the lack of (noncuspidal) rational points on certain modular curves of high level.  These curves parametrize isomorphism classes of elliptic curves with torsion data.  As the level of the torsion data increases, uniform boundedness requires that the corresponding modular curves all have positive genus; having genus at least $2$, i.e., being of general type, is desirable in view of Faltings' Theorem.  To investigate uniform boundedness of Brauer classes on K3 surfaces, one could start with the analogous purely geometric question: what is the Kodaira dimension of the moduli spaces of K3 surfaces that parametrize high-order Brauer class data?
			
			To have a reasonable moduli theory of K3 surfaces, e.g., to obtain coarse moduli spaces that are schemes of finite type over an algebraically closed field, one must first fix some polarization data.  Typically, one fixes a lattice $\Lambda$ of signature $(1,r)$ with $1\leq r \leq 19$ that can be primitively embedded in the K3 lattice $U^{\oplus 3}\oplus E_8(-1)^{\oplus 2}$ and considers K3 surfaces whose geometric N\'eron-Severi group $\NS \Xbar$ contains $\Lambda$.  To these lattice-polarized moduli spaces, one adds level structures coming from the Brauer group, e.g., a cyclic subgroup of order $n$ of the Brauer group.  McKinnie, Sawon, Tanimoto, and the first author show that when $n$ is prime these moduli spaces have $3$ or $4$ components, depending on $n$ and the polarization~\cite{MSTVA}.  One component is always isomorphic to a moduli space of K3 surfaces of higher degree, and, for some values of $n$, another is isomorphic to a moduli space of special cubic fourfolds.  
            Kodaira dimension estimates in~\cite{GHS-K3gentype,TVA-cubic4folds} show that such components are of general type for $n \gg 0$.  This has led the first author to make the following conjecture.
	    	
        	\begin{conj}[{\cite[Conj. 5.5]{VA-AWSnotes}}]\label{conjUniformBoundedness}
        		Fix a number field $k$ and a lattice $\Lambda$ together with a primitive embedding $\Lambda \hookrightarrow U^{\oplus 3}\oplus E_8(-1)^{\oplus 2}$. Let $X$ be a K3 surface over $k$ such that $\NS \Xbar \isom \Lambda$. Then there is a constant $B(k,\Lambda)$, independent of $X$, such that
        \[
        \#\left(\frac{\Br X}{\Br_0 X}\right) < B(k,\Lambda).
        \]
        	\end{conj}
Recent conditional work of the first author with Abramovich~\cite{AVA-Alevels,AVA-Vojta} is aimed at exploring the plausibility of the analogous conjecture for torsion on abelian varieties.

    \subsection{Outline}
        In~\S\ref{sec:GeometricallyKummer} we classify K3 surfaces $X$ such that $\NS \Xbar \isom \Lambda_d$.  In~\S\ref{sec:BrauerGroup}, we build on work of Skorobogatov and Zarhin~\cite{SZ-Kummer} to relate the Brauer group of such K3 surfaces and the Brauer group of principal homogeneous spaces of abelian surfaces to Galois equivariant homomorphisms between the $n$-torsion of isogenous non-CM elliptic curves.  Then, we prove results on Galois equivariant homomorphisms between the $n$-torsion of non-CM elliptic curves (\S\ref{sec:Hom}) and on related Galois equivariant endomorphisms (\S\ref{sec:End}).  Finally in~\S\ref{sec:UniformBounds}, we state and prove our main results.

    \subsection{Background and notation}
        For a rational prime $\ell$, we write $v_\ell$ for the corresponding $\ell$-adic valuation.  For any integer $n$, we let $n_{\textup{odd}}:= n\cdot2^{-v_2(n)}$ denote the odd part of $n$.  Given two positive integers $a$ and $b$, we let $\gcd(a,b^\infty) := \max_n\{\gcd(a,b^n)\}$.
	   
	    For any abelian group $G$ and positive integer $n$, we write $G_n$ for the subgroup of $n$-torsion elements and $G_{\textup{odd}}$ for the subgroup of elements of odd order.  If $G$ is finite, then we write $e(G)$ for its exponent and $\#G$ for its cardinality.
    
        Throughout, we assume that our fields have characteristic $0$.  For a field $k$, we let $\kbar$ denote a fixed algebraic closure and let $\Gamma_k$ denote the absolute Galois group $\Gal(\kbar/k).$  For any $\Gamma_k$-module $M$ and any quadratic character $\chi\colon\Gamma_k \to \mu_2$, we write
        \[
            M^{\chi} := \left\{m\in M : \sigma(m) = \chi(\sigma) m \textup{ for all }\sigma \in \Gamma_k \right\}.
        \]
        If $k'/k$ is a finite separable extension, then for any $i$ and any $\Gamma_k$-module $M$ we let $\Res_{k'/k}$ denote the restriction map $\HH^i(\Gamma_k, M) \to \HH^i(\Gamma_{k'}, M)$.

        For any smooth projective geometrically integral variety $X$ over a field $k$, we let $\Xbar$ denote the base change of $X$ to $\kbar$.  We let $\Pic X$ denote the Picard group of $X$, $\Pic^0 X$ denote the subgroup of $\Pic X$ that maps to the identity component of the Picard scheme, and $\NS X$ denote the quotient $\Pic X/\Pic^0 X$, the N\'eron-Severi group of $X$; we use $\sim$ to denote algebraic equivalence of divisors.  We write $\Br X := \HH^2_{\et}(X, \G_m)$ for the Brauer group of $X$, $\Br_1 X := \ker(\Br X \to \Br \Xbar)$ for the algebraic Brauer group of $X$, and $\Br_0 X := \im (\Br k \to \Br X)$ for the subgroup of constant algebras.
    
        Let $E$ be an elliptic curve over a field $k$.  We write $E_n$ for the subgroup scheme of $n$-torsion elements. A $k$-rational cyclic subgroup of $E$ is a cyclic subgroup $C\subset E$ such that $C^{\Gamma_k} = C$; note that $C$ is not necessarily contained in $E(k)$. If $\delta\in k^{\times}/k^{\times2}$, we write $E^{\delta}$ for the quadratic twist associated to $\delta\in k^{\times}/k^{\times2} \isom \HH^1(k, \mu_2) \to \HH^1(k, \Aut E)$.  For $\phi\colon E\to E'$ an isogeny between two elliptic curves, we write $\phi^{\vee}$ for the dual isogeny.  We call an isogeny \defi{cyclic} if its kernel is a cyclic group.
        
        Let $Y/k$ be a principal homogeneous space of an abelian variety $A$.  The period of $Y$, denoted $\per(Y)$, is the order of $Y$ in the Weil-Ch\^atelet group $\HH^1(k, A)$.

        \subsubsection{Kummer surfaces}\label{sec:ClassicalKummer}%

        	Let $A$ be an abelian surface over $k$, and let $f\colon Y \to A$ be a $2$-covering. Then $Z := f^{-1}(O)$ is a $k$-torsor for $A_2$, and the quotient $(A\times_k Z)/A_2$, where $A_2$ acts diagonally, is $k$-isomorphic to $Y$ as a $2$-covering. Thus, the antipodal involution $\iota\colon A \to A$ acts on $Y$ fixing $Z$ pointwise; elements of $Z$ give rise to the singular locus $S$ of the quotient $Y/\iota$. Let $X  := \Bl_{S}(Y/\iota)$ be the blow-up of $Y/\iota$ centered at $S$. The surface $X$ can also be constructed as the quotient of the blow-up $Y' := \Bl_{Z}(Y)$ by the involution $\iota'\colon Y' \to Y'$ induced by $\iota$.  We call $X$ the \defi{Kummer surface} associated to $Y$ (or $Z$) and denote it $\Kum Y$; it is a K3 surface.

        	Geometrically, the exceptional divisor of the blow-up map $X \to Y/\iota$ consists of $16$ pairwise disjoint smooth rational curves, each with self-intersection $-2$. This collection of curves gives a sublattice $\Z^{16}\subset \NS\Xbar$ whose saturation $\Lambda_K$ is called the \defi{Kummer lattice}. It is an even, negative-definite rank $16$ lattice of discriminant $2^6$ whose isomorphism type does not depend on the choice of $Y$; see~\cite{Nikulin}. There is an exact sequence of lattices
        	\begin{equation}
        		\label{eq:NScomparison}
        		0 \to \Lambda_K \to \NS\Xbar \to \NS\Ybar \to 0,
        	\end{equation}
        where the first map is the natural inclusion (see, e.g.,~\cite[Remark~2]{SZ-Kummer}).

        \subsubsection{Products of elliptic curves}\label{sec:KummerProducts}

        	Of particular interest to us is the case when $\Ybar$ is isomorphic to a product of two elliptic curves $E$, $E'$. Let $O$ and $O'$ be the respective origins of $E$ and $E'$. In addition to the Kummer lattice, the group $\NS\Xbar$ contains the classes $e := [E\times \{O'\}]$ and $e' := [\{O\}\times E']$. Define the lattice
        	\begin{equation}
        	\label{eq:Lprod}
        		\Lambda_{\prod} := \Lambda_K \oplus \langle e,e'\rangle \subseteq \NS\Xbar,
        	\end{equation}
        where $\oplus$ denotes an orthogonal direct sum and $\langle e,e'\rangle$ is isomorphic to a hyperbolic plane.  Note that $\Lambda_{\prod}$ does not depend on the particular choice of $E$ and $E'$.

        	Suppose next that $E$ and $E'$ are isogenous non-CM elliptic curves over $k$. Let $d$ be the smallest degree of a geometric isogeny $E \to E'$; such an isogeny is necessarily cyclic~\cite[Lemma 6.2]{MasserWustholz}. Define the lattice
        	\[
        		\Lambda_d := \NS(\Kum(\Ebar\times \Ebar')).
        	\]
        	The lattice $\Lambda_d$ depends on the integer $d$, but not on the specific isogeny $E\to E'$ of degree $d$. Indeed, the sequence~\eqref{eq:NScomparison} shows that $\Lambda_d$ is generated by some rational combinations of the $16$ classes of $(-2)$-curves, and pullbacks from $\NS(\Ebar\times \Ebar')$.  In turn, it is well known that the lattice $\NS(\Ebar\times \Ebar')$ is generated by the classes $e$, $e'$ and the class of the graph of a degree $d$ cyclic isogeny $E \to E'$. The intersection products of the graph with $e$ and $e'$ depend only on $d$, and not on the isogeny $E\to E'$.
           
		\begin{remark}
			Any isomorphisms with $\Lambda_K, \Lambda_{\prod},$ and $\Lambda_d$ are isomorphisms of abstract lattices; $\Lambda_K, \Lambda_{\prod},$ and $\Lambda_d$ do \emph{not} carry a Galois action.
		\end{remark}
    

		\section*{Acknowledgements}
		Question~\ref{question} was raised at the AIM workshop ``Brauer groups and obstruction problems: moduli spaces and arithmetic'' in March, 2013 for the particular case of K3 surfaces~\cite[Problem 1.5]{AIMProblemList}.  Lemma~\ref{lem:AlgBrBound} was proved at the same workshop in a discussion with Max Lieblich, Andrew Obus, and the second author, among others.  We thank the participants for allowing us to include this result.
   
		We thank Brendan Hassett for many fruitful discussions, in particular, for the suggestion that we focus on these $1$-parameter families.  We also thank Dan Abramovich, K\k{e}stutis \v{C}esnavi\v{c}ius, Max Lieblich, Ronen Mukamel, Drew Sutherland, and David Zureick-Brown for helpful conversations.  The proof of Theorem~\ref{thm:ellprimaryEC} was informed by a MathOverflow post by Jeremy Rouse~\cite{MO}.  

		We thank the anonymous referee for his or her careful reading of the paper and valuable comments.
   
		The first author was partially supported by NSF CAREER grant DMS-1352291; the second author was partially supported by NSA Young Investigator's Award \#H98230-15-1-0054 and NSF CAREER grant DMS-1553459.  This collaboration was partially supported by the Pacific Institute for Mathematical Sciences, Rice University, and the University of Washington; we thank these institutions for their support.

\section{Classification of certain rank $19$ K3 surfaces}
\label{sec:GeometricallyKummer}

	The goal of this section is to classify K3 surfaces $X$ over a {number field} $k$ such that $X/k\in\mathscr{K}_d$; {see Corollary~\ref{cor:Classification} below.}  As a by-product, we classify surfaces $Y/k \in \mathscr{A}_d^1$; see Proposition~\ref{prop:GeomProduct}.

	\begin{prop}\label{prop:GeometricallyKummerToKummer}
		There is a positive integer $M$ such that for any number field $k$, and any K3 surface $X/k$ with $\NS\Xbar$ {containing a sublattice isomorphic to $\Lambda_K$,} there is an extension $k_0/k$ of degree at most $M$ such that $X_{k_0}$ is a Kummer surface.
	\end{prop}
	\begin{remarks}\hfill
		
		\begin{enumerate}
			
			\item[(i)] The proof of Proposition~\ref{prop:GeometricallyKummerToKummer} actually proves much more -- it shows that $X_{k_0}$ is the Kummer surface of an abelian surface $A$ whose $2$-torsion is $k_0$-rational -- at the expense of a very large bound $M$.  It would be interesting to determine whether there is a smaller bound that yields the desired result and nothing stronger. \label{rmk:Stronger}
			\item[(ii)] The proposition can be generalized to any set of fields $\mathbb{K}$ for which there is a uniform bound on the degree of a field extension required to split an order $2$ element in $\Br k$, independent of the Brauer class and of the field $k\in \mathbb{K}$.  Such a bound exists for number fields by class field theory.
			\end{enumerate}
	\end{remarks}

    \begin{theorem}\label{thm:Classification}
    Let $X = \Kum Y$ be a Kummer surface over a field $k$ of characteristic $0$, with $Y \to A$ the $2$-covering associated to $X$. Assume that $X\in\mathscr{K}_d$ for some positive integer $d$.  Then there exist: a field extension $L/k$ of degree at most $12$, elliptic curves $E$ and $E'$ over $L$, and $\delta\in L^{\times}/L^{\times2}$,  such that
    \begin{enumerate}
        \item $A_L \isom E\times E'$, and
        \item there is a cyclic isogeny $\phi\colon E \to E'^{\delta}$ of degree $d$.
    \end{enumerate}  
    \end{theorem}

	\begin{cor}\label{cor:Classification}
	    There exists a positive integer $M_0$ such that, for all number fields $k$, all positive integers $d$ and all surfaces $X\in \mathscr{K}_d$, there exist: a field extension $L_0/k$ of degree at most $M_0$, elliptic curves $E$ and $E'$ over $L_0$, and a $2$-covering $Y\to E\times E'$, 
		such that
	    \begin{enumerate}
			\item $X_{L_0} \isom \Kum Y$, and
	        \item there is a cyclic isogeny $\phi\colon E \to E'$ of degree $d$.
	    \end{enumerate}  
	\end{cor}
	
	\begin{proof}
		This is immediate from Theorem~\ref{thm:Classification} and Proposition~\ref{prop:GeometricallyKummerToKummer}.
	\end{proof}

    We begin by proving Proposition~\ref{prop:GeometricallyKummerToKummer} in~\S\ref{sec:GeometricallyKummerToKummer}.  In~\S\ref{ss:GeomIsog}, we study non-CM elliptic curves that are geometrically isogenous; the results therein are used in~\S\ref{ss:GeomProds} to classify abelian surfaces $A$ such that $A\in \mathscr{A}_d^1$ for some $d$.  This is used in~\S\ref{sec:Proof} to prove Theorem~\ref{thm:Classification}.

    \subsection{Proof of Proposition~\ref{prop:GeometricallyKummerToKummer}}
	\label{sec:GeometricallyKummerToKummer}
	
	We use Nikulin's work on Kummer surfaces~\cite{Nikulin}, and the main idea of~\cite[Lemma~17.2.6]{Huybrechts}.  The group $\NS\Xbar$ is finitely generated; each of its generators can be represented by a curve which is defined by finitely many equations with finitely many coefficients.  Therefore, the image of the representation
	\[
		\rho\colon \Gamma_k \to O(\NS \Xbar) \hookrightarrow \GL_r(\Z),\quad r := \rank\NS \Xbar \leq 20
	\]
is finite.  Every finite subgroup of $\GL_r(\Z)$ injects into $\GL_r(\F_3)$.  Furthermore, $r\leq 20$ so $|\im(\rho)|$ can be bounded by an absolute constant.  Thus, for $H := \ker(\rho)$, the degree $[\kbar^H:k]$ can also be bounded by an absolute constant.

The hypothesis that $\NS\Xbar$ contains a sublattice isomorphic to $\Lambda_K$ ensures $\NS\Xbar$ contains $16$ pairwise disjoint $(-2)$-classes, represented over $\kbar$ by irreducible rational curves~\cite[Theorem~3]{Nikulin}.  We claim that these classes can be represented by curves defined over $\kbar^H$.  Let $[D]$ be such a class, with $D$ an irreducible rational curve over $\kbar$.  Since $\Pic \Xbar = \NS\Xbar = (\NS\Xbar)^H$, for $h \in H$, we have $h([D]) = [D]$, so $D$ and ${}^h\!D$ are linearly equivalent irreducible curves; in particular, $D$ and ${}^h\!D$ are integral elements of the linear system $|D|$.  Since $D$ is effective and $D^2 = {-2}$, the Riemann-Roch theorem for surfaces implies that $h^0(X,\scrO_X(D)) = 1$, so we must have ${}^h\!D = D$, which means that the irreducible curve $D$ is already defined over $\kbar^H$.

	Let $D_1,\dots,D_{16}$ be disjoint $(-2)$-curves on $X$, each defined over $\kbar^H$.  The class of the sum $\sum_i D_i$ is divisible by $2$ in $\Lambda_K \subseteq \NS\Xbar = (\NS \Xbar)^H$~\cite[\S3]{Nikulin}.  The Hochschild--Serre spectral sequence
	\begin{equation}\label{eq:HS}
		E_2^{p,q} := {\rm H}^p\big(H,{\rm H}_{\text{\'et}}^q\big(\Xbar,\G_m\big)\big) \Longrightarrow {\rm H}_{\text{\'et}}^{p+q}\big(X_{\kbar^H},\G_m\big)
	\end{equation}
gives rise to an exact sequence of low-degree terms 
	\[
		0 \to \NS X_{\kbar^H} \to (\NS \Xbar)^{H} = \NS \Xbar \to \Br \kbar^H.
	\]
This sequence shows that the obstruction to divisibility by $2$ of the class $\sum_i D_i$ in $\NS X_{\kbar^H}$ lies in $(\Br \kbar^H)_2$.  Since $\kbar^H$ is a number field, global class field theory shows that any nontrivial element of $(\Br \kbar^H)_2$ can be represented by a quaternion algebra over $\kbar^H$~\cite[Theorem 3.6]{Neukirch-BonnLectures}, and can thus be trivialized by a field extension ${k_0}/\kbar^H$ of degree at most $2$.  Thus, there is a divisor $D'$ on $X_{k_0}$ such that $\sum_i D_i \sim 2D'$, and hence a double cover morphism 
	\[
		\SPEC(\scrO_{X_{k_0}}\oplus \scrO_{X_{k_0}}(D')) =: \widetilde Y \to X_{k_0}
	\]
branched along $\sum_i D_i$.  The preimage of $\sum_i D_i$ consists of $16$ pairwise disjoint rational curves with self-intersection ${-1}$.  Contracting these curves gives rise to a surface $Y/{k_0}$.  The classification of surfaces (see, for example,~\cite{Beauville-Surfaces}) implies that $\Ybar$ is an abelian surface, and so $Y$ is a principal homogeneous space under its Albanese variety $A$.  There is an involution $\widetilde Y \to \widetilde Y$ associated to the double cover $\widetilde Y \to X_{k_0}$, which geometrically is the $[-1]$ map and whose fixed locus consists of the preimages of $\sum_i D_i$.  This involution in turn gives rise to an involution $Y \to Y$ whose fixed locus $Z$ consists of $16$ geometric points and is a torsor of $A_2$.  Thus $Y$ is a $2$-covering and $X_{k_0}$ is a Kummer surface.  As $[\kbar^H:k]$ is absolutely bounded, and $[k_0:k ] \leq 2[\kbar^H:k]$, the degree of $k_0/k$ is absolutely bounded.
\qed
        
		As pointed out in Remark~\ref{rmk:Stronger}(i), this proof actually shows that $X_{k_0}$ is a Kummer surface of an abelian surface $A$ such that $A_2\subset A(k_0)$.  Indeed, our determination of $\kbar^H$ implies that each $D_i$ is defined over $k_0$ and so $Z\subset Y$ consists of $16$ $k_0$-rational points.

    \subsection{Geometrically isogenous non-CM elliptic curves}
	\label{ss:GeomIsog}

    \begin{prop}\label{prop:twists}
        Let $E, E'$ be non-CM elliptic curves over a field $k$ of characteristic $0$.  Assume that there exists a cyclic isogeny $\overline{\phi}\colon \Ebar \to \Ebar'$.  Then $C := \ker \overline{\phi}$ is defined over $k$ and there exists a $\delta\in k^{\times}/k^{\times2}$ such that $E'^{\delta} \isom E/C$.  In particular, there exists a unique cyclic isogeny $\phi_{k(\sqrt{\delta})}\colon E_{k(\sqrt{\delta})} \to E'_{k(\sqrt{\delta})}$ that agrees with $\overline{\phi}.$  
    \end{prop}
    We begin with a lemma.
    \begin{lemma}\label{lem:KernelDeterminesIsoQuotient}
        Let $E$ be a non-CM elliptic curve over a field $k$ of characteristic $0$ and let $C, C' \subset E$ be $k$-rational cyclic subgroups of order $d$.  If $E/C$ and $E/C'$ are $k$-isomorphic, then $C = C'$.
    \end{lemma}
    \begin{proof}
        Let $\phi$ denote the quotient isogeny $E \to E/C$ and let $\psi$ denote the composition of the quotient $E\to E/C'$ with a $k$-isomorphism $E/C'\stackrel{\sim}{\to} E/C$.  Note that $\phi$ and $\psi$ are both isogenies and $\ker \phi = C$ and $\ker \psi = C'$.  
    
        Now consider $\psi^{\vee}\circ \phi \colon E \to E$.  This is an isogeny of degree $d^2$; since $E$ is non-CM, it must be equal to $\pm[d]$.  Hence $\psi = \pm \phi$ by the uniqueness of the dual isogeny and so $C'= \ker \psi = \ker \phi = C$.
    \end{proof}

    \begin{proof}[Proof of Proposition~\ref{prop:twists}]
        There is a Galois extension $L/k$ over which $C$ is defined.  Since $\overline{E}' \isom \overline{\left(E/C\right)}$ and $E'$ is non-CM, there exists a unique $\delta\in L^{\times}/L^{\times2}$ such that $E'^\delta_L \isom E_L/C$. Let $\sigma\in \Gal(L/k)$.  Then we have
        \[
            \overline{(E_L/C)} \isom\overline{E'^\delta_L} \isom \overline{E'^{\sigma(\delta)}_L} \isom  \overline{(E_L/\sigma(C))}.
        \]
        Thus, Lemma~\ref{lem:KernelDeterminesIsoQuotient} {applied over $\kbar$ yields that} $C = \sigma(C)$ for all $\sigma \in \Gal(L/k)$.  Thus $C$ is $k$-rational, and we may take $L=k$ above.  This completes the proof.
    \end{proof}

    \subsection{Geometric products of elliptic curves}\label{ss:GeomProds}

        \begin{prop}\label{prop:GeomProduct}
           Let $A$ be an abelian surface over a field $k$ such that $A/k\in \mathscr{A}_d^1$ for some positive integer $d$.  Then there exist: an extension $L/k$ with $[L:k] \in \{1,2,3,4,6, 8, 12\}$, elliptic curves $E$ and $E'/L$, and $\delta\in L^{\times}/L^{\times2}$ such that 
           \begin{enumerate}
               \item $A_L \isom E \times E'$, and
               \item there exists a cyclic $L$-isogeny $\phi\colon E \to {E'}^{\delta}$ of degree $d$.
            \end{enumerate}
        \end{prop}
		
        \begin{lemma}\label{lem:product}
              Let $A$ be an abelian surface over an algebraically closed field such that $\NS A$ contains a hyperbolic plane. Then $A$ is isomorphic to a product of elliptic curves.  In addition,
			  \begin{itemize}
				  \item if $\rank \NS A = 2$, then the elliptic curves are not isogenous,
				  \item if $\rank \NS A = 3$, then the elliptic curves are non-CM, isogenous, and the degree of a cyclic isogeny between them is $\frac{1}{2}\disc\NS A$, and
				  \item if $\rank \NS A = 4$, then the elliptic curves are isogenous and CM.
			\end{itemize}
        \end{lemma}
        
        \begin{proof}
			Let $e_1$ and $e_2$ generate a hyperbolic plane contained in $\NS A$, so $e_1^2 = e_2^2 = 0$ and $e_1\cdot e_2 = 1$. Let $D$ be a divisor on $A$ representing the class $e_1 + e_2$. By~\cite[Corollary~2.2(b)]{Kani}, either $D$ or ${-D}$ is ample, because $D^2 > 0$; assume without loss of generality that $D$ is ample. Using~\cite[Proposition~2.3]{Kani}, we conclude that for $i = 1$ and $2$, the class $e_i$ is represented by a multiple $m_iE_i$ of an elliptic curve $E_i$, and $m_i > 0$ since $E_i\cdot D > 0$.  Moreover, since $1 = e_1\cdot e_2 = m_1m_2\, (E_1\cdot E_2) \in m_1m_2\Z$, we have $m_i = 1$.  Translating the curves $E_1$ and $E_2$ if necessary, we may assume they are elliptic subgroups of $A$.  We may thus define a morphism of abelian varieties $\phi\colon A \to A/E_1 \times A/E_2$ using projections.  Let $(P + E_1, Q + E_2)\in A/E_1 \times A/E_2$.  Since $P + E_1\sim E_1$ and $Q + E_2 \sim E_2$, and $E_1\cdot E_2 = 1$, there exists a unique point $R\in A$ such that $\phi(R) = (P + E_1, Q + E_2)$.  Hence, $\phi$ is an isogeny of degree $1$, i.e., an isomorphism.  
			
			Note that we have group isomorphisms
			\[
				\NS A \isom \NS(E_1 \times E_2)\isom \Z\cdot[E_1]\oplus \Z\cdot[E_2]\oplus \Hom(E_1,E_2).
			\]
			This proves all the remaining statements except for the computation of the degree of a cyclic isogeny $E_1 \to E_2$ when $\rank \NS A = 3$.  Assume that $\rank \NS A = 3$. Then $E_1$ and $E_2$ are non-CM, and we have $\Hom(E_1,E_2) = \langle\phi\rangle$, where $\phi\colon E_1 \to E_2$ is a cyclic isogeny.  The image of $\phi$ in $\NS(E_1 \times E_2)$ is the graph of $\phi$ and the images of $[E_1], [E_2]$ are $E_1 \times O$ and $O\times E_2$ respectively.  Hence, a simple calculation shows that $\disc \NS(E_1\times E_2) = 2\deg(\phi)$.
        \end{proof}

        \begin{proof}[Proof of Proposition~\ref{prop:GeomProduct}]
            Part $(2)$ of the proposition follows from $(1)$, Lemma~\ref{lem:product}, and Proposition~\ref{prop:twists}.  Hence, we reduce to proving $(1)$.
            
            Let $\calM_{1,1} = \A^1$ denote the coarse moduli space of elliptic curves, parametrized by their $j$-invariant. The coarse moduli space $\calA_2$ of principally polarized abelian surfaces contains the Humbert surface $\calH_1 := \Sym^2 \calM_{1,1} \isom \A^2$, which is the locus of abelian surfaces with product structure. By Lemma~\ref{lem:product}, there exist non-CM geometrically isogenous elliptic curves $E$ and $E'$  over $\kbar$ such that $\Abar \isom E \times E'$. Thus the surface $\Abar$ gives rise to a point $x \in \calH_1(\kbar)$, with coordinates $(j(E) + j(E'), j(E)\cdot j(E')) \in \A^2(\kbar)$. For any $\sigma \in \Gal(\kbar/k)$, we have $\sigma(\Abar) = \Abar$, so $E\times E' \isom \sigma(E\times E')$, and thus $x \in \calH_1(k)$. Therefore $j(E) + j(E')$ and $j(E)\cdot j(E')$ belong to $k$, and so there is an extension $k_0/k$ of degree at most $2$ such that $j(E), j(E')\in k_0$.
            
            Since $j(E), j(E')\in k_0$, we may assume that $E, E'$ are defined over $k_0$.  By Proposition~\ref{prop:twists}, after possibly replacing $E'$ with a quadratic twist, we may assume that $E$ and $E'$ are isogenous over $k_0$; let $\psi\colon E \to E'$ be a cyclic isogeny defined over $k_0$ and let $d := \deg(\psi)$.  Since $\Abar \isom \Ebar\times \Ebar'$, $A_{k_0}$ is a twist (as an abelian surface) of $E\times E'$, so corresponds to an element of $\HH^1(k_0, \Aut(\Ebar\times \Ebar')).$  The abelian surface $A_{k_0}$ is isomorphic to a product of elliptic curves if 
            \[
                [A_{k_0}] \in \im\left(
                \HH^1(k_0, \Aut(\Ebar))\times \HH^1(k_0, \Aut(\Ebar'))
                \to \HH^1(k_0, \Aut(\Ebar\times \Ebar'))
                \right).
            \]
            Thus to prove $(1)$ it suffices to show that for any element $\varphi\in \HH^1(k_0, \Aut(\Ebar\times\Ebar'))$, there exists an extension $L/k_0$ with $[L:k_0]\in\{1,2,3,4,6\}$ such that
            \[
            \varphi_L \in \im\left(
            \HH^1(L, \Aut(\Ebar))\times \HH^1(L, \Aut(\Ebar'))
            \to \HH^1(L, \Aut(\Ebar\times \Ebar'))
            \right).
            \]

            Since $E$ and $E'$ are $k_0$-isogenous non-CM elliptic curves, we have
            \begin{equation}\label{eq:PresentationAut}
                \End(\Ebar\times\Ebar') = 
                \begin{pmatrix}
                    \End(\Ebar) & \Hom(\Ebar', \Ebar)\\
                    \Hom(\Ebar, \Ebar') & \End(\Ebar')
                \end{pmatrix} = 
                \begin{pmatrix}
                    \Z & \Z\psi^{\vee}\\
                    \Z\psi & \Z
                \end{pmatrix}
                \isom
                \begin{pmatrix}
                    \Z & \Z\sqrt{d}\\
                    \Z\sqrt{d} & \Z
                \end{pmatrix} \subset \Mat_2(\R)
                ;
            \end{equation}
            in particular, $\Gal(\kbar/k_0)$ acts trivially on $\End(\Ebar\times \Ebar')$ and $\Aut(\Ebar\times\Ebar')$ is isomorphic to a subgroup of $\GL_2(\R)$.  Since $\Gal(\kbar/k_0)$ acts trivially, we have
            \[
                \HH^1(k_0, \Aut(\Ebar\times \Ebar')) = 
                \Hom_{\textup{cts}}(\Gal(\kbar/k_0), \Aut(\Ebar\times\Ebar')).
            \]
            
            Let $\varphi\in \Hom_{\textup{cts}}(\Gal(\kbar/k_0), \Aut(\Ebar\times \Ebar'))$ and let $L$ be the fixed field of $\varphi^{-1}\left(\{\pm I\}\right)$.  Then $\im \textup{Res}_{L/k_0}(\varphi)\subset \{\pm I\}$, so
            \[
           \textup{Res}_{L/k_0}(\varphi) \in \im\left(
            \HH^1(L, \Aut(\Ebar))\times \HH^1(L, \Aut(\Ebar'))
            \to \HH^1(L, \Aut(\Ebar\times \Ebar'))
            \right).
            \]
            Since $[L:k] = [L:k_0]\cdot [k_0:k]\leq 2[L:k_0]$, it remains to show that $[L:k_0]\in\{1,2,3,4,6\}$.  Note that $\varphi$ yields an isomorphism $\Gamma_{k_0}/\Gamma_{L} \to \im\varphi/(\{\pm I\}\cap(\im\varphi)),$ so
            \[
                [L:k_0] = \#\Gal(L/k_0) = \#\frac{\im\varphi}{\{\pm I\}\cap(\im\varphi)}.
            \]
           The image of $\varphi$ must be a profinite subgroup of the discrete group $\Aut(\Ebar\times\Ebar')$, hence must be finite.  {Furthermore, by~\eqref{eq:PresentationAut}, every element of $\Aut(\Ebar\times \Ebar')$ has integer trace and determinant.  }Thus, the following lemma completes the proof.
        \end{proof}

            \begin{lemma}\label{lem:FiniteGroups}
                {Let $G$ be a finite subgroup of $\GL_2(\R)$ such that $\Tr(x), \det(x)\in \Q$ for all $x\in G$.  Then $G$ is isomorphic to a cyclic group of order $1,2,3,4,$ or $6$ or a dihedral group of order $4, 6, 8,$ or $12$.  In addition, if $G$ is isomorphic to a dihedral group of order $4, 8$ or $12$, then $-I\in G$.}
            \end{lemma}  
            \begin{proof}
                Since every compact group of $\SL_2(\R)$ is conjugate to a subgroup of $\textup{SO}_2(\R)$~\cite[Theorem~6]{Iwasawa}, $G^+ := G\cap \SL_2(\R)$ is conjugate to a finite subgroup of $\textup{SO}_2(\R)\isom S^1$, so is cyclic.  Therefore, $G$ is either cyclic or an extension of $\Z/2\Z$ by a cyclic group.
                
                Let $x\in G$.  Since the trace and determinant of $x$ are rational numbers and the minimal polynomial of $x$ divides $T^{|G|} - 1$, the characteristic polynomial of $x$ must be one of the following
                \[
                    (T - 1)^2, \; (T + 1)^2, \; T^2 - 1, \; T^2 + 1, \;
                    T^2 + T + 1, \; T^2 - T + 1.
                \]
                Therefore, any $x\in G$ has order equal to $1,2,3,4$ or $6$, and if $x\in G\setminus G^+$, then $x$ has order $2$.  Hence $G$ is a cyclic group of order $1,2,3,4$ or $6$ or it is a dihedral group of order $4,6,8$, or $12$.  
                
                Assume that $G$ is a dihedral group of order $4, 8$ or $12$.  Then $G^+$ is a cyclic group of even order, hence contains an element $g$ of order $2$.  Since $g\in G^+$, the minimal polynomial of $g$ must be a proper divisor of $T^2 - 1$, and so $g = -I.$
            \end{proof}

    \subsection{Proof of Theorem~\ref{thm:Classification}}\label{sec:Proof}

			The exact sequence~\eqref{eq:NScomparison} and the isomorphism $\NS\Xbar\isom \Lambda_d$ imply that $A\in \mathscr{A}_d^1$. Then
			 Proposition~\ref{prop:GeomProduct} completes the proof.
		\qed

\section{Bounds on the Brauer group in terms of Galois-invariant homomorphisms}%
\label{sec:BrauerGroup}

	Let $Y/k$ be a surface such that $Y\in \mathscr{A}_d^N$ for some integers $N$ and $d$, so in particular $Y$ is a principal homogeneous space under $A := \Alb Y$.    Since, as lattices, $\NS \Ybar \isom \NS \Abar$,  Proposition~\ref{prop:GeomProduct} gives a field extension $L/k$ of degree at most $12$, and a pair of elliptic curves $E$ and $E'$ over $L$ such that $A_L\isom E\times E'$.  
	
	As a first step towards bounding the $n$-torsion of the group $\Br Y/\Br_1 Y$, we show that the group $(\Br Y/\Br_1 Y)_n$ can be embedded into a group determined by Galois-invariant homomorphisms associated to $E$ and $E'$, after possibly extending the ground field to a field whose degree is bounded by a constant depending only on the period of $Y$. More precisely, we prove the following theorem.

	\begin{thm}\label{thm:BrAbHom}
        Let $n$ be a positive integer and let $Y/k$ be a surface in $\bigcup_{d,N}\mathscr{A}_{d}^N$. Let  $L$, $E$, and $E'$ be as above. There is a field extension $k'/k$ with $[k':k]\leq (\gcd(\per(Y),n^{\infty}))^4$ such that 
		 \[
            \frac{(\Br Y)_n}{(\Br_1 Y)_n} \hookrightarrow
            \frac{\Hom_{L'}(E_n, E_n')}{\left(\Hom(\Ebar, \Ebar')/n\right)^{\Gamma_{L'}}},
		\]
		where $L'$ denotes the compositum of $k'$ and $L$.  If $n$ is relatively prime to $\per(Y)$ and $L= k$, then this injection is an isomorphism.  
	\end{thm}
	
    Work of Skorobogatov and Zarhin allows us to extend Theorem~\ref{thm:BrAbHom} to  Kummer surfaces with geometric N\'eron-Severi lattice isomorphic to $\Lambda_d$.  Indeed, if $X = \Kum Y$ is a  Kummer surface over a field $k$ of characteristic $0$ and $\pi \colon \tilde{Y}\to X$ is the associated degree $2$ quotient map then the  proof of~\cite[Thm.~2.4]{SZ-Kummer} shows that for any positive integer $n$, the map $\pi^*$ induces an injection 
            \begin{equation}\label{eq:KummerToPHS}
                \frac{(\Br X)_n}{(\Br_1 X)_n} \hookrightarrow 
                \frac{(\Br Y)_n}{(\Br_1 Y)_n},
            \end{equation}
            which is an isomorphism if $n$ is odd.  (Although~\cite[Thm.~2.4]{SZ-Kummer} assumes that $Y$ is an abelian surface, the proof holds under the weaker assumption that $\Ybar$ is an abelian surface.)  Since $Y$ has period at most $2$ and $\NS \Ybar\isom\begin{pmatrix} 0 & 1  & 1\\1 & 0 & d\\1 & d & 0\end{pmatrix}$ by~\eqref{eq:NScomparison}, we conclude the following.
	\begin{cor}\label{cor:BrHom}
         Let $X = \Kum Y$ be a Kummer surface over a field $k$ of characteristic $0$ such that $X\in \mathscr{K}_d$ for some $d$, and let $L$, $E$, and $E'$ be as above. 
		 There is a field extension $k'/k$ with $[k':k]\leq 2^4$ such that 
		 \[
            \frac{(\Br X)_n}{(\Br_1 X)_n} \hookrightarrow
            \frac{\Hom_{L'}(E_n, E_n')}{\left(\Hom(\Ebar, \Ebar')/n\right)^{\Gamma_{L'}}},
		\]
		where $L'$ denotes the compositum of $k'$ and $L$.  If $n$ is odd or $Y$ is the trivial $2$-covering, then we may take $k' = k$.  If $n$ is odd and $L=k' = k$, then this injection is an isomorphism.  
		\qed
	\end{cor}
	
	In \S\ref{ss:Br-p-covers} we show that if $f\colon Y \to A$ is an $N$-covering of an abelian variety, and if $n$ is coprime to $N$, then $(\Br Y)_n/(\Br_1 Y)_n$ is isomorphic to $(\Br A)_n/(\Br_1 A)_n$. In \S\ref{ss:BrProds}, we use~\cite[Prop.~3.3]{SZ-Kummer} to show that for an abelian surface $A$ as above, the group $(\Br A)_n/(\Br_1 A)_n$ injects into a quotient of Galois invariant homomorphisms between the $n$-torsion of two elliptic curves.  We then combine these ingredients in~\S\ref{sec:ProofBrAbHom} to give a proof of Theorem~\ref{thm:BrAbHom}.

    \subsection{Brauer groups of $N$-coverings of abelian varieties}
	\label{ss:Br-p-covers}%

        \begin{prop}\label{prop:Br2covering}
            Let $N$ be a positive integer, let $Y$ and $Y'$ be principal homogeneous spaces of an abelian variety $A$, both defined over $k$, and let $f\colon Y' \to Y$ be an $N$-covering.  Then for any integer $n$ that is coprime to $N$, we have 
            \[
                f^*\colon\frac{(\Br Y)_n}{(\Br_1 Y)_n}\stackrel{\sim}{\To} \frac{(\Br Y')_n}{(\Br_1 Y')_n}.
            \]
        \end{prop}

        \begin{proof}
            Since $f$ is a $N$-covering, it is dominant and finite of degree $N^{2g}$.  Thus by~\cite[Prop 1.1 and Lemma 1.3]{ISZ-DiagonalQuartics}, $f^*\colon \Br \Ybar \to \Br \Ybar'$ is surjective with kernel killed by $N^{2g}$, and the kernel of $f^*\colon \Br Y \to \Br Y'$ is also killed by $N^{2g}$.  This completes the proof of injectivity.  We continue with the proof of surjectivity.
        
            The above facts already imply that $f^*$ gives an isomorphism of Galois modules $(\Br \Ybar)_n\isom (\Br \Ybar')_n$ for any integer $n$ coprime to $N$.  Combining the above with~\cite[Prop. 1.3]{CTS-Transcendental}, we obtain the following commutative diagram with exact rows:
    		\[
    			\xymatrix{
    				(\Br Y)_n \ar[r]\ar@{^{(}->}[d]^{f^*} & 
                    (\Br \Ybar)_n^{\Gamma} \ar@{=}[d]^{f^*} \ar[r]^(.4){\delta} &
                    \HH^2(k, \Pic \Ybar)_n\ar[d]^{f^*} \\
    				(\Br Y')_n \ar[r] & 
                    (\Br \Ybar')_n^{\Gamma}  \ar[r]^(.4){\delta} &
                    \HH^2(k, \Pic \Ybar')_n \\
    				}
    		\]
        
            Let $\alpha \in (\Br Y')_n$ and let $\overline\beta\in (\Br \Ybar)_n^{\Gamma}$ be such that $f^*\overline\beta = \overline\alpha := \alpha\otimes_k{\kbar}$.  Since the diagram has exact rows and commutes, this implies that $f^*\delta(\overline\beta) = 0$.  We claim that the kernel of $f^*\colon \HH^2(k, \Pic \Ybar) \to \HH^2(k, \Pic \Ybar')$ is an $N$-primary group.  If so, then $\delta(\overline\beta) = 0$ and so $\overline\beta = \beta\otimes_{k}\kbar$ for some $\beta\in (\Br Y)_n$.  Hence $f^*\beta = \alpha + \alpha_0$ for some $\alpha_0\in (\Br_1 Y')_n$, which proves surjectivity.  
            
            It remains to prove the claim.  We factor $f^*\colon \Pic \Ybar \to \Pic \Ybar'$ as $\Pic \Ybar\onto \im f^* \into \Pic \Ybar'$ and consider the two short exact sequences
            \[
                0 \to \ker f^* \to \Pic \Ybar \to \im f^* \to 0 
                \quad \textup{and}\quad
                0\to \im f^* \to \Pic \Ybar' \to \coker f^*\to 0.
            \]
            Since $\coker f^*$ and $\ker f^*$ are $N$-primary, both $\HH^1(k,\coker f^*)$ and $\HH^2(k,\ker f^*)$ are also $N$-primary.  Thus, by the long exact sequences in cohomology associated to the above two short exact sequences, the kernel of the composition
            \[
                \HH^2(k, \Pic \Ybar) \to \HH^2(k, \im f^*) \to \HH^2(k, \Pic \Ybar')
            \]
            is also $N$-primary, which completes the proof of the claim.
        \end{proof}

    \subsection{Brauer groups of abelian surfaces that are geometrically 
    products}\label{ss:BrProds}
        
        Let $A$ be an abelian surface over $k$ such that $A\in\mathscr{A}_d^1$ for some $d$.  Then by Proposition~\ref{prop:GeomProduct}, there exists a field extension $L/k$ with $[L:k]\leq 12$ and elliptic curves $E$ and $E'/L$ such that $A_L = E\times E'$.  Then~\cite[Prop. 3.3]{SZ-Kummer} almost immediately yields the following proposition.
        
        \begin{prop}\label{prop:BrProduct}
            Let $A$ be an abelian surface over $k$ such that  $A\in\mathscr{A}_d^1$ for some $d$ and let $E$ and $E'$ be the associated elliptic curves over the extension $L$ as above.  Then for all $n$ we have an injection
            \[
                \frac{(\Br A)_n}{(\Br_1 A)_n} \into 
                \frac{\Hom_{\Gamma_L}(E_n, E'_n)}{\left(\Hom(\Ebar, \Ebar')/n\right)^{\Gamma_L}}.
            \]
            If $L = k$, then this injection is an isomorphism.
        \end{prop}
        \begin{proof}
            By~\cite[Prop. 3.3]{SZ-Kummer}, we have a canonical isomorphism
            \[
                \frac{(\Br A_L)_n}{(\Br_1 A_L)_n} \isom 
                \frac{\Hom_{\Gamma_L}(E_n, E'_n)}
				{\left(\Hom(\Ebar, \Ebar')/n\right)^{\Gamma_L}}.
            \]
            Composing with the injection ${(\Br A)_n}/{(\Br_1 A)_n} \into
				  {(\Br A_L)_n}/{(\Br_1 A_L)_n}$ induced by $\Res_{L/k}$ 
			completes the proof.
        \end{proof}

    \subsection{Proof of Theorem~\ref{thm:BrAbHom}}\label{sec:ProofBrAbHom}
        Recall that $A = \Alb Y$.  For any positive integer $n$, $\HH^1(k, A_n)$ surjects onto $\HH^1(k, A)_n$, therefore $Y$ can be realized as a $\per(Y)$-covering $f\colon Y \to A$.  Let $N := \gcd(\per(Y), n^{\infty})$.  Then $f$ can be factored as an $N$-covering $f'\colon Y' \to A$ and a covering $f''\colon Y \to Y'$ whose degree is relatively prime to $n$.  Since $Y'$ is an $N$-covering of $A$, there exists a field extension $k'/k$ with $[k':k]\leq\#A_{N} = N^4$ such that $Y'_{k'}\isom A_{k'}$.  We claim that we have the following inclusions and isomorphisms
		\[
			\frac{(\Br Y)_n}{(\Br_1 Y)_n} 
			\stackrel{(f''^*)^{-1}}{\To} 
			\frac{(\Br Y')_n}{(\Br_1 Y')_n} \into
			\frac{(\Br Y'_{k'})_n}{(\Br_1 Y'_{k'})_n}\stackrel{\sim}{\To}
			\frac{(\Br A_{k'})_n}{(\Br_1 A_{k'})_n}\into
			\frac{\Hom_{L'}(E_n, E'_n)}{\left(\Hom(\Ebar, \Ebar')/n\right)^{\Gamma_{L'}}}.
		\]
		From left to right, the first map is an isomorphism by Proposition~\ref{prop:Br2covering}, the second is an inclusion by the definition of the algebraic Brauer group, the third map is an isomorphism because $Y'_{k'}\isom A_{k'}$, and the last is an inclusion by~Proposition~\ref{prop:BrProduct}.  Furthermore, by Proposition~\ref{prop:BrProduct} the last injection is an isomorphism when $L = k$.
		\qed

\section{Galois-equivariant homomorphisms between $E_n$ and $E'_n$}
\label{sec:Hom}

	Let $E$ and $E'$ be non-CM elliptic curves over a field $k$ with a geometric cyclic isogeny of degree $d$ between them.  Then by Proposition~\ref{prop:twists}, there is $\delta \in k^\times/k^{\times 2}$ such that $E$ and $E'^\delta$ are $k$-isogenous.
	
	\begin{thm}\label{thm:MainHomToEnd}
		Let $E$ and $E'$ be geometrically isogenous non-CM elliptic curves over a field $k$ of characteristic $0$; let $d$ and $\delta$ be as above. For any positive integer $n$ we have the following divisibility relations for cardinalities and exponents.
		\begin{enumerate}
			\item If $\delta\in k^{\times2}$, then 
			\begin{align*}
	            \#\left(\frac{\Hom_{k}\left(E_n, E'_n\right)}{\left(\Hom(\Ebar, \Ebar')/n\right)^{\Gamma_k}}\right) &\left.\middle|\; \gcd(d,n)\cdot
	            \#\left(\frac{\End_{k}\left( E'_n\right)}{\left(\End(\Ebar')/n\right)^{\Gamma_k}}\right),\right. \quad \textup{and}\\
	            e\left(\frac{\Hom_{k}\left(E_n, E'_n\right)}{\left(\Hom(\Ebar, \Ebar')/n\right)^{\Gamma_k}}\right) &\left.\middle|\; \gcd(d,n)\cdot
	            e\left(\frac{\End_{k}\left( E'_n\right)}{\left(\End(\Ebar')/n\right)^{\Gamma_k}}\right).\right.
			\end{align*}
			\item If $\delta\not\in k^{\times2}$, then
			\begin{align*}
				\left.\#\left(\frac{\Hom_k\left(E_n, E'_n\right)}
				{\left(\Hom(\Ebar, \Ebar')/n\right)^{\Gamma_k}}\right) 
				\;\middle|\right. &
				\gcd(2,n)^4\cdot \gcd(d,n)^2\cdot
				\#\left(\frac{\End_{k(\sqrt{\delta})}(E'^{\delta}_n)}
				{\End_{k}(E'^{\delta})}\right), \quad \textup{and}\\
				\left.e\left(\frac{\Hom_k\left(E_n, E'_n\right)}
				{\left(\Hom(\Ebar, \Ebar')/n\right)^{\Gamma_k}}\right) 
				\;\middle|\right. &
				\gcd(2,n)\cdot \gcd(d,n)^2\cdot
				e\left(\frac{\End_{k(\sqrt{\delta})}(E'^{\delta}_n)}
				{\End_{k}(E'^{\delta})}\right).
			\end{align*}
		\end{enumerate}
		
	\end{thm}
	
    Parts $(1)$ and $(2)$ of Theorem~\ref{thm:MainHomToEnd} are proved in~\S\ref{sec:RatlIsogeny} and~\S\ref{sec:GeometricIsogeny}, respectively.
    \subsection{Galois action on isogenies}\label{sec:Isogenies}

	\begin{lemma}\label{lem:GeometricHom}
		Let $E$ and $E'$ be non-CM geometrically isogenous elliptic curves over a field $k$ of characteristic $0$.  Then 
		\[
			\left(\Hom(\Ebar, \Ebar')/n\right)^{\Gamma_k} \isom 
			\begin{cases}
				\Z/n\Z &\textup{if }E, E' \textup{ are isogenous over }k\\
				\Z/\gcd(2,n)\Z & \textup{otherwise.}
			\end{cases}
		\]
	\end{lemma}
	\begin{proof}
		Since $\Ebar$ and $\Ebar'$ are non-CM isogenous curves, $\Hom(\Ebar, \Ebar') = \Z\phi$ where $\phi$ is a cyclic $\kbar$-isogeny.  By Proposition~\ref{prop:twists} the action of $\Gamma_k$ on $\Hom(\Ebar, \Ebar')$ factors through $\Gal(k(\sqrt{\delta})/k)$ for some $\delta\in k^{\times}/k^{\times2}$ and $\delta\in k^{\times2}$ if and only if $E$ and $E'$ are $k$-isogenous.  In the case where $\delta\notin k^{\times2}$, then Proposition~\ref{prop:twists} implies that $\phi$ is the composition of a $k$-rational cyclic isogeny $\phi'\colon E \to E'^{\delta}$ with a $k(\sqrt{\delta})$-isomorphism $E'^{\delta} \to E'$.  In particular, the nontrivial element of $\Gal(k(\sqrt{\delta})/k)$ acts on $\Hom(\Ebar, \Ebar')$ by multiplication by $-1$.  Since 
		\[
			(\Z/n\Z)^{(x\mapsto-x)} = \textstyle{\frac{n}{\gcd(2,n)}}\Z/n\Z,
		\]
		 this completes the proof.
	\end{proof}
    \subsection{A $k$-rational cyclic isogeny}\label{sec:RatlIsogeny}
    
    \begin{prop}\label{prop:HomToEndRatl}
        Let $n$ be a positive integer and let $E, E'$ be non-CM elliptic curves over a field $k$ of characteristic $0$.  Assume that there exists a degree $d$ cyclic $k$-isogeny $\phi\colon E\to E'$.  Then composition with $\phi^{\vee}$ induces a homomorphism
        \[
            \frac{\Hom_{k}\left(E_n, E'_n\right)}{\left(\Hom(\Ebar, \Ebar')/n\right)^{\Gamma_k}}
            \stackrel{(-\circ\phi^{\vee})}{\To}
            \frac{\End_{k}\left( E'_n\right)}{\left(\End(\Ebar')/n\right)^{\Gamma_k}},
        \]
		where the kernel is a cyclic group of order dividing $m := \gcd(d,n)$ and the cokernel is $m$-torsion.
    \end{prop}
    \noindent Theorem~\ref{thm:MainHomToEnd}(1) is a corollary of this proposition.
    
    \medskip
    
    To prove the Proposition, we begin with two lemmas.
	\begin{lemma}\label{lem:ComposingWithIsogeny}
		Let $E, E', E''$ be elliptic curves, let $g\colon E \to E'$ be a cyclic isogeny of degree $d$, and let $n,n'$ be positive integers such that {$n'\gcd(d,n) =\gcd(dn', n)$}.  Then $f\in \Hom(E_{n}, E''_{n})$ factors through $[n']\circ g$ if and only if $f\circ g^{\vee}\circ[n/\gcd(dn',n)] = 0\in \Hom(E'_{n}, E''_{n})$.
	\end{lemma}
	\begin{proof}
		 Assume that $f$ factors through $[n']\circ g$, i.e., that $f = h \circ[n']\circ g$ for some $h\in Hom(E'_{n}, E''_{n})$.  Then
		 \[
		 	f\circ g^{\vee}\circ\left[\frac{n}{\gcd(dn',n)}\right] = 
			h \circ[n']\circ g\circ g^{\vee}\circ\left[\frac{n}{\gcd(dn',n)}\right] =
			h \circ \left[\frac{dn'}{\gcd(dn',n)}\right]\circ [n] = 0
		 \]
		 in $\Hom(E'_{n}, E''_{n})$.
		 
		 Now let $f\in \Hom(E_{n}, E''_{n})$ be such that $f\circ g^{\vee}\circ[n/\gcd(dn',n)] = 0\in \Hom(E'_{n}, E''_{n})$.  We will show that $\ker f$ contains $\ker ([n']\circ g)$ which implies that $f$ factors through $[n']\circ g$.  By definition of the dual isogeny, $ E'_d/\ker g^{\vee} \isom g^{\vee}(E'_d)=\ker (g\colon E \to E') $.  Similarly, since $n'\gcd(d,n) =\gcd(dn', n)$, we have
		 \[
		 	\ker \left([n']\circ g \colon E_n \to E'_n\right) = 
			g^{\vee}(E'_{n'\gcd(d, n)}) = 
			\left(g^{\vee}\circ\left[\textstyle{\frac{n}{\gcd(dn',n)}}\right]\right)(E'_n).
		 \]
		 Since $f\circ g^{\vee}\circ[n/\gcd(dn',n)] = 0$, this gives the desired containment.
	\end{proof}
	
	\begin{lemma}\label{lem:HomToEnd}
		Let $n$ be a positive integer and let $E,E'$ be elliptic curves over $k$.  Assume that there exists a cyclic degree $d$ $k$-isogeny $\phi\colon E\to E'$.  Then there is a short exact sequence of $\Gamma_k$-modules
		\[
			0 \to \Hom(\phi(E_{m}), E'_{m}) 
			\stackrel{-\circ\phi\circ\left[\frac{n}{m}\right]}{\To}
			\Hom(E_{n}, E'_{n})
			\stackrel{-\circ\phi^{\vee}}{\To}
			H 
			\to 0,
		\]
		where $m = \gcd(d,n)$ and $H := \left\{f\in \End(E'_{n}) : f \circ\phi\circ[n/m] = 0\right\}$.  
	\end{lemma}
    
	\begin{proof}
		The first map is clearly injective.  Thus to prove the lemma, it suffices compute the kernel and image of
		\[
			(-\circ\phi^{\vee})\colon \Hom(E_{n}, E'_{n}) \to \End(E'_{n}).
		\]
		An application of Lemma~\ref{lem:ComposingWithIsogeny} with $n' = 1$ and $g = \phi^{\vee}$ implies that the image of $(-\circ\phi^{\vee})$ is equal to $H.$  Another application of Lemma~\ref{lem:ComposingWithIsogeny} with $n' = n/m$ and $g = \phi$ shows that the kernel of $(-\circ\phi^{\vee})$ is equal to the image of $\Hom(\phi(E_{m}), E'_{n})$, which completes the proof.
	\end{proof}
    \begin{proof}[Proof of Proposition~\ref{prop:HomToEndRatl}]
		We consider the following diagram.
		\[\xymatrix{
			0\ar[r]& \frac{n}{m}\Z/n\Z \ar[r]\ar@{^{(}->}[d]&
			 \Hom_{k}(\phi(E_m), E'_m) \ar[r]\ar@{^{(}->}[d]&
			 \ker(-\circ\phi^{\vee})\ar@{^{(}->}[d]\\
			0\ar[r] & 
			\left(\Hom(\Ebar,\Ebar')/n\right)^{\Gamma_k}
			\ar[r]\ar[d]^{(-\circ\phi^{\vee})} & 
			\Hom_{k}(E_n, E'_n) 
			\ar[r]\ar[d]^{(-\circ\phi^{\vee})} & 
			\frac{\Hom_{k}\left(E_n, E'_n\right)}{\left(\Hom(\Ebar, \Ebar')/n\right)^{\Gamma_k}}
			\ar[r] \ar[d]^{(-\circ\phi^{\vee})}& 0\\
			0 \ar[r] & 
			\left(\End(\Ebar')/n\right)^{\Gamma_k} \ar[r]\ar@{->>}[d] &
			\End_{k}(E'_n) \ar[r]\ar@{->>}[d]& 
			\frac{\End_{k}\left( E'_n\right)}{\left(\End(\Ebar')/n\right)^{\Gamma_k}} \ar[r] \ar@{->>}[d]
			& 0\\
			& \Z/m\Z \ar[r] & \frac{\End_{k}(E'_n)}{\Hom_{k}(E_n, E'_n)\circ\phi^{\vee}}\ar[r]& \coker(-\circ\phi^{\vee})\ar[r]& 0
		}
		\]
		We claim that this diagram is commutative and has exact rows and columns.  The commutativity and exactness of the middle two rows is clear.  Since $E$ and $E'$ are non-CM and $\phi$ is a $k$-isogeny, $\Hom(\Ebar, \Ebar') = \Z\phi$ and $\End(\Ebar') = \Z$ as Galois modules.  Thus the leftmost column is exact.  The rightmost column is exact by definition. Taking $\Gamma_k$-invariants of the exact sequence in Lemma~\ref{lem:HomToEnd} gives exactness of the middle column.  Since the maps in the top and bottom rows are induced by the maps in the middle two, the remaining squares commute. 		
		
		The proposition is concerned with the rightmost column of the above diagram; we must prove that $\ker(-\circ\phi^{\vee})$ is a cyclic group of order dividing $m$ and that $\coker(-\circ\phi^{\vee})$ is $m$-torsion.  Let $f\in \End_k(E'_n)$ and let $a\in \Z$ be an inverse of $d/m$ modulo $n$; then
		\[
			[m]\circ f = f\circ[m] = f\circ[ad] = (f\circ[a]\circ\phi)\circ\phi^{\vee}\in\Hom_k(E_n, E'_n)\circ\phi^{\vee}.
		\]
		Hence $\End_k(E'_n)/(\Hom_k(E_n, E'_n)\circ\phi^{\vee})$ is $m$-torsion and thus so is $\coker(-\circ\phi^{\vee})$.
		
		Now consider $\Hom_k(\phi(E_m), E_m')$.  Since $m\mid d$, we have $\phi(E_m)$ is a cyclic group of order $m$. Thus after fixing a generator of $\phi(E_m)$, we have $\Hom \left(\phi(E_m), E'_m\right)\isom E'_m\isom (\Z/m\Z)^2$ as abelian groups.  Therefore, the quotient $\Hom_{k}\left(\phi(E_m), E'_m\right) / ({\frac{n}{m}}\Z/n\Z)$ is a subgroup of the cyclic group of order $m$.  
		
		We claim that the bottom leftmost horizontal arrow is injective.  If so, then the snake lemma implies that
		\[
			\ker(-\circ\phi^{\vee})\isom \frac{\Hom_k\left(\phi(E_m), E'_m\right)}{\textstyle{\frac{n}{m}}\Z/n\Z} \subset \Z/m\Z,
		\] 
		as desired. To prove the claim, it suffices to show that for $b\in \Z$, $[b]\in \End(E'_n)$ is in the image of $(-\circ\phi^{\vee})$  only if $m\mid b$.  By Lemma~\ref{lem:HomToEnd}, $[b]$ is in the image of $(-\circ\phi^{\vee})$ if and only if $[b]\circ\phi\circ\left[\frac{n}{m}\right] = \phi\circ\left[\frac{bn}{m}\right] = 0$ on $E_n$.  This happens if and only if $m \mid b$.
	\end{proof}

    \subsection{A geometric, non $k$-rational isogeny}
    \label{sec:GeometricIsogeny}
	
		In this section we consider the case of two non-CM geometrically isogenous elliptic curves $E, E'$ over $k$ that are \emph{not} $k$-isogenous.  By Proposition~\ref{prop:twists}, there exists a unique nontrivial $\delta\in k^{\times}/k^{\times2}$ such that $E$ and $E'^{\delta}$ are $k$-isogenous.  
	
		\begin{prop}\label{prop:HomToEndNonRational}
			Let $E$ and $E'$ be non-CM geometrically isogenous elliptic curves over $k$, that are not isogenous over $k$ and let  $\delta\in k^\times/k^{\times2}$ be such that $E$ and ${E'}^{\delta}$ are $k$-isogenous.  Let $d$ be the degree of a cyclic isogeny between $E$ and ${E'}^{\delta}$. Then there is a homomorphism
			\[
				\frac{\Hom_k(E_n, E'_{n})}
				{\left(\Hom(\Ebar, \Ebar')/n\right)^{\Gamma_k}}
				\to
				\frac{\End_{k(\sqrt{\delta})}({E'}^{\delta}_n)}{\End_{k}({E'}^{\delta}_n)}
			\]
			whose kernel has order at most $\gcd(2,n)^4\gcd(d,n)^2$ and exponent at most $\gcd(2,n)\gcd(d,n)^2.$
		\end{prop}
		\noindent Theorem~\ref{thm:MainHomToEnd}(2) is a corollary of this proposition.

		\medskip

		The homomorphism in Proposition~\ref{prop:HomToEndNonRational} is obtained as the composition of two homomorphisms:
		\[
			\frac{\Hom_k(E_n, E'_{n})}
			{\left(\Hom(\Ebar, \Ebar')/n\right)^{\Gamma_k}}
			\to
			\frac{\Hom_{k(\sqrt{\delta})}(E_n, E'^{\delta}_n)}
			{\Hom_{k}(E_n, E'^{\delta}_n)}
			\quad\textup{and}\quad
			\frac{\Hom_{k(\sqrt{\delta})}(E_n, E'^{\delta}_n)}{\Hom_{k}(E_n, E'^{\delta}_n)} \to
			\frac{\End_{k(\sqrt{\delta})}(E'^{\delta}_n)}{\End_{k}(E'^{\delta}_n)}.
		\]
		We first study each homomorphism individually.
		
		\begin{prop}\label{prop:HomToEndNonRationalHelper}
			Let $E$ and $E'$ be non-CM elliptic curves over $k$.  Assume that there exists a nontrivial $\delta\in k^\times/k^{\times2}$ such that $E$ and $E'^{\delta}$ are isogenous over $k$.  Then there is an exact sequence.
			\[
				0 \to 
				\frac{\Hom_k(E_n, E'_{\gcd(2,n)})}
				{\left(\Hom(\Ebar, \Ebar')/n\right)^{\Gamma_k}}
				\to
				\frac{\Hom_k(E_n, E'_{n})}
				{\left(\Hom(\Ebar, \Ebar')/n\right)^{\Gamma_k}}
				\to
				\frac{\Hom_{k(\sqrt{\delta})}(E_n, E'^{\delta}_n)}
				{\Hom_{k}(E_n, E'^{\delta}_n)}
			\]
		\end{prop}
		\begin{proof}
			First let us verify that each quotient is defined.  This is immediate for the middle and right quotient.  The left quotient is defined by Lemma~\ref{lem:GeometricHom}.
			
			Over $k(\sqrt\delta)$, $E'$ and $E'^{\delta}$ are isomorphic.  Fix a $k(\sqrt{\delta})$-isomorphism  $f\colon E' \stackrel{\sim}{\To} E'^{\delta}$.  Then we have a homomorphism 
			\[
			\Hom_k(E_n, E'_{n})
			\to
			\frac{\Hom_{k(\sqrt{\delta})}(E_n, E'^{\delta}_n)}
			{\Hom_{k}(E_n, E'^{\delta}_n)},
			\quad \psi \mapsto f\circ\psi.
			\]
			The kernel of this homomorphism consists of all maps $\psi\in \Hom_k(E_n, E'_{n})$ such that $f\circ\psi$ is $\Gamma_k$-invariant.  If $\sigma\in \Gamma_{k(\sqrt{\delta})}$, then by assumption $f^\sigma = f$ and $\psi^\sigma = \psi$.  Let $\sigma\in \Gamma_k\setminus \Gamma_{k(\sqrt{\delta})}$; then $(f\circ \psi)^{\sigma} = f^{\sigma}\circ\psi^{\sigma} = -f\circ \psi.$
			Thus, if $(f\circ \psi)^{\sigma} = f\circ \psi$, we have $-f\circ\psi = f\circ\psi.$  Since $f$ is an isomorphism, this implies that $2\psi = 0$, or equivalently that $\im \psi \subset E'_{\gcd(2,n)}$.  Thus, we have an exact sequence.
			\[
				0 \to 
				\Hom_k(E_n, E'_{\gcd(2,n)})
				\to
				\Hom_k(E_n, E'_{n})
				\to
				\frac{\Hom_{k(\sqrt{\delta})}(E_n, E'^{\delta}_n)}
				{\Hom_{k}(E_n, E'^{\delta}_n)}
			\]
			Lemma~\ref{lem:GeometricHom} then completes the proof.
		\end{proof}
	
		\begin{prop}\label{prop:NonRationalHomBound}
			Let $n$ be a positive integer, let $E, \Etilde$ be non-CM elliptic curves over a field $k$, let $k'/k$ be a quadratic extension, and let $\chi$ be the quadratic character associated to $k'/k$.  Assume that there exists a cyclic $k$-isogeny $\phi\colon E\to \Etilde$, let $d = \deg(\phi)$, and let $m = \gcd(d,n)$.  Then composition with $\phi^{\vee}$ induces a homomorphism
			\[
				\frac{\Hom_{k'}(E_n, \Etilde_n)}{\Hom_{k}(E_n, \Etilde_n)} \to
				\frac{\End_{k'}(\Etilde_n)}{\End_{k}(\Etilde_n)}.
			\]
			Furthermore, the kernel $K$ of this homomorphism  fits in an exact sequence
			\[
				0 \to 
				\frac{\Hom_{k'}(\phi(E_m),\Etilde_m)}
				{\Hom_{k}(\phi(E_m),\Etilde_m)} \to K \to
				\Hom_{k'}(\phi(E_m), \Etilde_m)^{\chi} 
			\]
			so $K$ is a finite abelian group with order dividing $m^2\gcd(m,2)$ and exponent dividing $m^2$.
		\end{prop}
		\begin{proof}
			Let $H := \left\{f\in \End(\Etilde_{n}) : f \circ\phi\circ[n/m] = 0\right\}$.  By Lemma~\ref{lem:HomToEnd}, we have an exact sequence of $\Gamma_k$-modules
			\[
				0 \to \Hom(\phi(E_m), \Etilde_m) \to \Hom(E_n, \Etilde_n) \to H \to 0.
			\] 
			By taking the long exact sequences in Galois cohomology for $\Gamma_k$ and $\Gamma_{k'}$, we have the following commutative diagram with exact rows and columns.
			\[
			\xymatrix{
				& 
				\Hom_{k}(\phi(E_m), \Etilde_m)
				\ar@{^{(}->}[r]^{\Res_{k'/k}} \ar@{^{(}->}[d] & 
				\Hom_{k'}(\phi(E_m), \Etilde_m)\ar@{^{(}->}[d]\\
				0 \ar[r] & 
				\Hom_{k}(E_n, \Etilde_n) \ar[r]^{\Res_{k'/k}} 
				\ar[d]^{(-\circ\phi^\vee)} & 
				\Hom_{k'}(E_n, \Etilde_n) \ar[r]\ar[d]^{(-\circ\phi^\vee)} &
				\frac{\Hom_{k'}(E_n, \Etilde_n)}{\Hom_{k}(E_n, \Etilde_n)}
				\ar[r]\ar[d]^{(-\circ\phi^\vee)}& 0\\
				0 \ar[r]& 
				H^{\Gamma_k} \ar[r]^{\Res_{k'/k}}\ar[d]^{\partial_k} & 
				(H)^{\Gamma_{k'}} \ar[r]\ar[d]^{\partial_{k'}} &
				\frac{(H)^{\Gamma_{k'}}}{H^{\Gamma_k}} \ar[r] & 0\\
				& \HH^1(\Gamma_k, \Hom(\phi(E_m), \Etilde_m))
				\ar[r]^{\Res_{k'/k}} &
				\HH^1(\Gamma_{k'}, \Hom(\phi(E_m), \Etilde_m))}
			\] 
			Since the inclusion $H \hookrightarrow \End(\Etilde_n)$ induces an inclusion $(H)^{\Gamma_{k'}}/H^{\Gamma_k} \hookrightarrow {\End_{k'}(\Etilde_n)}/{\End_{k}(\Etilde_n)}$, $K$ is the kernel of the right-most vertical map in the above diagram. Hence, the snake lemma applied to the middle two rows yields
			\[
				0 \to 
				\frac{\Hom_{k'}(\phi(E_m),\Etilde_m)}
				{\Hom_{k}(\phi(E_m),\Etilde_m)} \to K \to
				\ker\left(\Res_{k'/k}\colon \im \partial_k \to \im \partial_{k'}\right)\to 0.
			\] 
			By the inflation-restriction exact sequence, $\ker\left(\Res_{k'/k}\colon \im \partial_k \to \im \partial_{k'}\right)$ is contained in $\HH^1(\Gamma_k/\Gamma_{k'}, \Hom_{k'}(\phi(E_m), \Etilde_m)).$  The cocycle condition implies that the nontrivial element of $\Gamma_k/\Gamma_{k'}$ must be sent to an element of $\Hom_{k'}(\phi(E_m), \Etilde_m)^{\chi} $
			which yields the desired exact sequence for $K$.
            
            Since $m\mid d$, we have a group isomorphism $\phi(E_m) \isom \Z/m\Z$, and so $\Hom (\phi(E_m), \Etilde_m)$ is (non-canonically) isomorphic to $\Etilde_m\isom (\Z/m\Z)^2$.  On the other hand, the natural inclusion $\phi(E_m) \subset \Etilde_m$, together with composition by multiplication by an integer, shows that $\Hom_{k}(\phi(E_m),\Etilde_m)$ contains $\Z/m\Z$.  Thus $\Hom_{k'}(\phi(E_m),\Etilde_m)/\Hom_{k}(\phi(E_m),\Etilde_m)$ is a cyclic group of order dividing $m$.  
            
            Now we consider $\Hom_{k'}(\phi(E_m), \Etilde_m)^{\chi}$.
            We have already shown that $\Hom (\phi(E_m), \Etilde_m)\isom (\Z/m\Z)^2$.  The natural inclusion $\phi(E_m) \subset \Etilde_m$ and its compositions with multiplication by an integer yield a $\Gamma_k$-homomorphism $\Z/m\Z\hookrightarrow \Hom_{k'}(\phi(E_m), \Etilde_m)$.  However, the only such maps that satisfy $\psi^{\sigma} = -\psi$ are those in the image of the subgroup $\frac{m}{\gcd(m,2)}\Z/m\Z$.  Thus, $\Hom_{k'}(\phi(E_m), \Etilde_m)^{\chi} $ is a subgroup of $\Z/\gcd(m,2)\Z \times \Z/m\Z$.  Hence, $K$ is an abelian extension of  a subgroup of $\Z/\gcd(m,2)\Z \times \Z/m\Z$ by a subgroup of $\Z/m\Z$.
		\end{proof}
		
		\begin{proof}[Proof of Proposition~\ref{prop:HomToEndNonRational}]
			Consider the homomorphism $\varphi$ obtained as the composition of
			\[
				\frac{\Hom_k(E_n, E'_{n})}
				{\left(\Hom(\Ebar, \Ebar')/n\right)^{\Gamma_k}}
				\stackrel{\varphi_1}{\to}
				\frac{\Hom_{k(\sqrt{\delta})}(E_n, E'^{\delta}_n)}
				{\Hom_{k}(E_n, E'^{\delta}_n)}
				\quad\textup{and}\quad
				\frac{\Hom_{k(\sqrt{\delta})}(E_n, (E')^{\delta}_n)}{\Hom_{k}(E_n, E'^{\delta}_n)} \stackrel{\varphi_2}{\to}
				\frac{\End_{k(\sqrt{\delta})}((E')^{\delta}_n)}{\End_{k}(E'^{\delta}_n)},
			\]
			where $\varphi_1$ is as in Proposition~\ref{prop:HomToEndNonRationalHelper} and $\varphi_2$ is as in Proposition~\ref{prop:NonRationalHomBound}.  Hence $\ker \varphi$ is an extension of a subgroup of $\ker \varphi_2$ by $\ker \varphi_1$ and so
            \[
            \#(\ker\varphi) \mid \#(\ker\varphi_1)\#(\ker\varphi_2)
            \quad\textup{and}\quad
                e(\ker\varphi) \mid e(\ker\varphi_1)e(\ker\varphi_2).
            \]  
            The exponent and order of $\ker\varphi_2$ are bounded by Proposition~\ref{prop:NonRationalHomBound}, so it remains to study $\ker \varphi_1$. By Proposition~\ref{prop:HomToEndNonRationalHelper}, 
            \[
                \ker \varphi_1 = \frac{\Hom_k(E_n, E'_{\gcd(2,n)})}
				{\left(\Hom(\Ebar, \Ebar')/n\right)^{\Gamma_k}}.
            \]
            Since $E'_{\gcd(2,n)} \isom (\Z/\gcd(2,n)\Z)^2$, a choice of $2$ generators for $E_n$ determines a  group isomorphism $\Hom(E_n, E'_{\gcd(2,n)})\isom {E'}_{\gcd(2,n)}^2 \isom (\Z/\gcd(2,n)\Z)^4$.  Hence by Lemma~\ref{lem:GeometricHom}, $\ker \varphi_1$ is a subgroup of $(\Z/\gcd(2,n)\Z)^3$.
		\end{proof}
		
\section{Elliptic curves and abelian representations}\label{sec:End}
            
	\subsection{Galois-equivariant endomorphisms of $E_n$}%
    Let $n$ be a positive integer and let $E$ be an elliptic curve over a field $k$ of characteristic $0$.  Let
        \[
        	\rho_{E,n}\colon \Gamma_k \to \Aut(E_{n}) 
            \isom \GL_2(\Z/n\Z)
        \]
        denote the Galois representation coming from the action of Galois on the $n$-torsion of $E$, and let $G_{E, n}$ denote the quotient of the image of $\rho_{E,n}$ modulo scalar matrices.
        
        The image of $\rho_{E,n}$ determines the ring of Galois equivariant endomorphisms of $E_{n}$, namely $\End_k(E_{n})$ is the subring of $\End(E_{n})$ that commutes with all elements of $\rho_{E,n}(\Gamma_k)$.  
        
        \begin{prop}\label{prop:ranks}
        Let $\ell$ be a prime, let $s$ be a positive integer, and let $E$ be an elliptic curve over a field $k$ of characteristic $0$.  Then
            \[
                \dim_{\F_{\ell}}
                \frac{\End_{k}(E_{\ell^s})}
                {\End_{k}(E_{\ell^{s-1}})\circ[\ell]}
                =
                \begin{cases}
                    4& \textup{if } G_{E, \ell^s} = \{1\},\\
                    2& \textup{if } G_{E, \ell^s} \neq \{1\} \textup{ and $\im(\rho_{E,\ell^s})$ is abelian, and}\\
                    1& \textup{if }\im(\rho_{E,\ell^s})\textup{ is non-abelian.}
                \end{cases}
            \]
        \end{prop}
        \begin{cor}\label{cor:End}
			Let $E$ be an elliptic curve over $k$ and let $n$ be a positive integer.  Then we have an isomorphism of abelian groups
			\[
				\End_{k}(E_n) \isom 
                \Z/n\Z \times \Z/n_1\Z \times \left(\Z/n_2\Z\right)^2 
			\]
			for positive integers $n_2|n_1|n$.  Furthermore, $n_1$ is the largest integer dividing $n$ such that $\Gal(k(E_{n_1})/k)$ is abelian and $n_2$ is the largest integer dividing $n$ such that $\Gal(k(E_{n_2})/k) \subset (\Z/n_2\Z)^{\times}$ where $a\in (\Z/n_2\Z)^{\times}$ acts by $P\mapsto aP$.  If $E$ is non-CM, then $(\End_k(\Ebar)/n)^{\Gamma_k} \isom \Z/n\Z$ and hence
			\[
				\frac{\End_k(E_n)}{(\End_k(\Ebar)/n)^{\Gamma_k}} \isom \Z/n_1\Z \times \left(\Z/n_2\Z\right)^2.
			\]
		\end{cor}
		\begin{proof}
            Note that $k(E_n)$ is the compositum of $k(E_{\ell^{n(\ell)}})$ where $n(\ell) = v_{\ell}(n)$.  Thus, we have an inclusion
            \[
                \Gal(k(E_n)/k) \hookrightarrow \prod_{\ell}\Gal(k(E_{\ell^{n(\ell)}})/k) \isom \prod_{\ell} \im \rho_{E,\ell^{n(\ell)}}.
            \]
            Furthermore, for each prime $\ell$ we have a surjection $\Gal(k(E_n)/k) \to \Gal(k(E_{\ell^{n(\ell)}})/k) \isom  \im \rho_{E,\ell^{n(\ell)}}.$  Hence, $\Gal(k(E_n)/k)$ is abelian if and only if $\Gal(k(E_{\ell^{n(\ell)}})/k)$ is abelian for all primes $\ell$ and $\im \rho_{E, n} \subset (\Z/n\Z)^{\times}$ if and only if $\im \rho_{E, \ell^{n(\ell)}} \subset (\Z/\ell^{n(\ell)}\Z)^{\times}$ for all primes $\ell$.
            
            Since $\End_{k}(E_n) \subset \Mat_2(\Z/n\Z)$, the fundamental theorem of finitely generated abelian groups implies that
            \[
                \End_{k}(E_n) \isom \Z/n_0\Z \times \Z/n_1\Z \times \Z/n_2\Z \times \Z/n_3\Z
            \]
            for unique positive integers $n_3|n_2|n_1|n_0|n$. Note that $\End_{k}(E_n)$ always contains the image of $\Z \into \End(E) \to \End(E_n)$, so $n_0 = n$.  Additionally, for any prime $\ell$ and any positive integer $s$ with $\ell^s\mid n$, Proposition~\ref{prop:ranks} shows that 
            \begin{enumerate}
                \item $\ell^s|n_2 = n_3$ if and only if $G_{E,\ell^s} = \{1\}$, and
                \item $\ell^s|n_1$ if and only if $\im \rho_{E, \ell^s}$ is abelian.
            \end{enumerate}
            Combining these facts with the above arguments completes the proof.
		\end{proof}

		\begin{cor}\label{cor:End2}
			Let $E$ be an elliptic curve over $k$ and let $n$ be a positive integer.  Let $k'/k$ be a field extension. Then we have an isomorphism of abelian groups
			\[
				\frac{\End_{k'}(E_{n})}{\End_{k}(E_{n})} \isom \Z/{\textstyle \frac{n_1'}{n_1}}\Z \times \left(\Z/{\textstyle \frac{n_2'}{n_2}}\Z\right)^2,
			\]
			where $n_1'$ (respectively $n_1$) is the largest integer dividing $n$ such that $\Gal(k'(E_{n_1'})/k')$ (respectively $\Gal(k(E_{n_1})/k)$) is abelian and $n_2'$ (respectively $n_2$) is the largest integer dividing $n$ such that $\Gal(k'(E_{n_2'})/k') \subset (\Z/n_2'\Z)^{\times}$ (respectively $\Gal(k(E_{n_2})/k) \subset (\Z/n_2\Z)^{\times}$).
		\end{cor}
        \begin{proof}
            This follows from Corollary~\ref{cor:End}.
        \end{proof}
        The following lemma will be useful in the proof of Proposition~\ref{prop:ranks}.
            \begin{lemma}\label{lem:CommutingMatrices}
                Let $\ell$ be a prime, let $s$ be a positive integer, and let $A\in \Mat_2(\Z/\ell^s\Z)$ be a non-central element.  Let $\mu <  s$ be the maximal non-negative integer such that $A \in \Z I \bmod {\Mat_2(\ell^\mu\Z/\ell^s\Z)}$.  Then the ring
                \[
                    \frac{{\{M \in \M_2(\Z/\ell^s\Z) : AM = MA\}}}{
                    \{M \in \M_2(\ell^{s-\mu}\Z/\ell^s\Z) : AM = MA\}}.
                \]
                is generated by $I$ and a matrix $A'$ such that $\ell^{\mu}A' - A \in \Z I$.
            \end{lemma}
            \begin{proof}
                Let $A=\begin{pmatrix} \alpha & \beta \\ \gamma& \delta\end{pmatrix}\in \Mat_2(\Z/\ell^s\Z)$.  Since $A$ is not in the center of $\Mat_2(\Z/\ell^s\Z)$, at least one of $\alpha- \delta, \beta$, or $\gamma$ is nonzero modulo $\ell^s$ and so the quantity $\min(v_{\ell}(\alpha - \delta), v_{\ell}(\beta), v_{\ell}(\gamma))$ is well-defined, and is equal to $\mu < s$.  An elementary algebra calculation shows that a matrix $M = \begin{pmatrix} a & b \\ c & d\end{pmatrix}\in \M_2(\Z/\ell^s\Z)$ commutes with $A$ if and only if $(a - d, b, c)^T$ is in the kernel of
                \[
                    \begin{pmatrix}
                        \beta & \delta - \alpha & 0 \\
                        \gamma & 0  & \delta - \alpha \\
                        0 & \gamma & -\beta
                    \end{pmatrix}.
                \]

                Let $v_0, v_1, v_2\in \Z/\ell^s\Z$ be such that $(\ell^{\mu}v_0, \ell^{\mu}v_1, \ell^{\mu}v_2) = (\alpha - \delta, \beta, \gamma)$.  Note that $v_0, v_1$ and $v_2$ are unique modulo $\ell^{s-\mu}$ and at least one of $v_0, v_1, v_2$ is nonzero modulo $\ell$.  Then   $(v_0,v_1,v_2)^T$ is in the kernel of the above $3\times3$ matrix and indeed generates the kernel modulo $(\ell^{s - \mu}\Z/\ell^s\Z)^3$.  Thus, any matrix $M$ that commutes with $A$ is such that $\ell^{\mu}M \in \Z A + \Z I$.  This completes the proof.
            \end{proof}

        \begin{proof}[Proof of Proposition~\ref{prop:ranks}]
            After fixing a basis for $E_{\ell^s}$, we identify $\End(E_{\ell^s})$ with $\M_2(\Z/\ell^s\Z)$ and $\Aut(E_{\ell^s})$ with $\GL_2(\Z/\ell^s\Z)$. 
             If $G_{E, \ell^s} = \{1\}$, then $\End_{k}(E_{\ell^s}) = \End(E_{\ell^s})$ so the result is immediate.  
			 
			 Now assume that the image of $\rho_{E, \ell^s}$ is nonabelian, and let $M_1$ and $M_2$ be two elements in the image that do not commute with each other.  Suppose that $\End_{k}(E_{\ell^s})/\left(\End_{k}(E_{\ell^{s-1}})\circ[\ell]\right)$ is an $\F_{\ell}$-vector space of dimension greater than $1$ and let $A\in \Mat_2(\Z/\ell^s\Z)$ be a matrix representing a non-scalar element.  
			  Since $A\bmod \ell$ is not a scalar matrix, Lemma~\ref{lem:CommutingMatrices} implies that $M_1$ and $M_2$ must both be linear combinations of $I$ and $A$.  
			  As $M_1$ and $M_2$ do not commute with each other, this gives a contradiction.  Hence, $\End_{k}(E_{\ell^s})/\left(\End_{k}(E_{\ell^{s-1}})\circ[\ell]\right)$ consists only of scalar matrices. 
			  
             It remains to consider the case that $G_{E,\ell^s} \neq \{1\}$ and that the image of $\rho_{E, \ell^s}$ is abelian.  Fix an element $A\in \im \rho_{E, \ell^s}$ that does not reduce to the identity in $G_{E,\ell^s}$.  Since $\im \rho_{E, \ell^s}$ is abelian, Lemma~\ref{lem:CommutingMatrices} implies that $\im \rho_{E, \ell^s}\mod \M_2(\ell\Z/\ell^s\Z)$ is contained in the ring generated by $I$ and $A'$ where $\ell^{\mu}A' - A \in \Z I$ for a maximal $\mu$.  Hence,
             \[
                 \frac{\End_{k}(E_{\ell^s})}
                 {\End_{k}(E_{\ell^{s-1}})\circ[\ell]} = \frac{{\{M \in \M_2(\Z/\ell^s\Z) : A'M = MA'\}}}{\{M \in \M_2(\ell\Z/\ell^s\Z) : A'M = MA'\}}.
             \]
             Another application of Lemma~\ref{lem:CommutingMatrices} then shows that this quotient is $2$-dimensional.
            \end{proof}

	\subsection{Abelian subgroups of $\GL_2(\Z/\ell^s\Z)$}\label{ss:conjugacy}

            \begin{prop}\label{prop:IndexAbSubgroups}
            	Let $s$ be a positive integer, and put $s' = \lceil s/2\rceil$.  Let $\ell$ be a prime, and let $H < \GL_2(\Z/\ell^s\Z)$ be an abelian subgroup.  There exists $A' \in \M_2(\Z/\ell^s\Z)$ with $A' \bmod \ell\not\in \langle I\rangle$ such that the group $H \bmod \M_2(\ell^{s'}\Z/\ell^{s}\Z)$ is contained in the ring $\langle I, A' \bmod \ell^{s'}\Z\rangle$.  Consequently, $\#H \leq \ell^{3s}$.
            \end{prop}
            
            \begin{proof}
            	For $A \in H$, define $\mu_A  = \max\{ \mu : A \equiv \textup{scalar matrix} \bmod \ell^\mu\}$, and let $\mu = \min\{\mu_A : A \in H\}$.  Pick $A \in H$ such that $\mu_A = \mu$, and let $A' \in \M_2(\Z/\ell^s\Z)$ be a matrix such that $\ell^\mu A' - A \in \Z I$.  If $\mu < s/2$ then $s-\mu \geq s'$ and so Lemma~\ref{lem:CommutingMatrices} shows that $H \bmod M_2(\ell^{s - \mu}\Z/\ell^s\Z)$ is contained in the ring $\langle I, A' \rangle$.  On the other hand, if $\mu \geq s'$ then for every $A \in H$, the matrix $A \bmod \ell^{s'}\Z$ is a scalar matrix, by definition of $\mu$, and so is contained in the ring $\langle I, A' \rangle$ for any matrix $A'\in \Mat_2(\Z/\ell^s\Z)$.
	
		Now we show the bound $\#H \leq \ell^{3s}$ holds. Let $H'$ be the group $H \bmod \M_2(\ell^{s'}\Z/\ell^{s}\Z)$. Reduction modulo $\ell^{s'}$ gives rise to an exact sequence of finite abelian groups
		\[
			0 \to K \to H \to H' \to 0,
		\]
		where $K$ is the kernel of reduction.  Our work above shows that an element of $H'$ is a linear combination of $I$ and $A'$ with coefficients in $\Z/\ell^{s'}\Z$, and hence $\#H' \leq (\ell^{s'})^2$.  To bound the order of $K$, we note that $K$ is contained in the kernel of the reduction map $\GL_2(\Z/\ell^s\Z) \to \GL_2(\ell^{s'}\Z/\ell^s\Z)$.  An element of this kernel has the form $I + M$ for some $M \in \M_2(\ell^{s'}\Z/\ell^s\Z)$. Hence $\#K \leq (\ell^{s - s'})^4$. Putting these facts together, we obtain 
		\[
		\#H \leq (\ell^{s'})^2\cdot (\ell^{s - s'})^4 = (\ell^{s})^2(\ell^{s - s'})^2 \leq (\ell^{s})^2(\ell^{\lfloor s/2 \rfloor})^2 \leq \ell^{3s}.\tag*{$\qed$}
		\]
		\hideqed
            \end{proof}

	       \begin{cor}\label{cor:Conjugacy}
	           Let $\ell$ be an odd prime, let $s$ be a positive integer, let $s' := \lceil \frac{s}2\rceil$, and let $H<\GL_2(\Z/\ell^s\Z)$ be an abelian subgroup.  Then $H \bmod \Mat_2(\ell^{s'}\Z/\ell^s\Z)$ is conjugate to a subgroup of one of the following groups. 
	           \begin{align*}
	               C_{s}(\ell^{s'}) := &\left\{ \begin{pmatrix} x & 0 \\0 & y \end{pmatrix} :  x,y \in (\Z/\ell^{s'}\Z)^{\times}  \right\}\\
	               C_{ns}^{t, \overline{\varepsilon}}(\ell^{s'}) := &
                   \left\{ 
                   \begin{pmatrix} x & \varepsilon\ell^ty \\y & x \end{pmatrix}
                    :  (x,y) \in (\Z/\ell^{s'}\Z)^2, x^2 - \varepsilon \ell^t y^2 \notin\ell\Z/\ell^{s'}\Z 
                    \right\},\ 0 \leq t\leq {s'} - 1,\\
                    & \overline{\varepsilon} \in \left(\Z/\ell^{{s'}-t}\Z\right)^{\times}\big /\left(\Z/\ell^{{s'}-t}\Z\right)^{\times2},\ \varepsilon\in (\Z/\ell^{s'}\Z)^{\times} \textup{ fixed such that }
                    \varepsilon\mapsto\overline{\varepsilon},\\
					& \textup{ and }\varepsilon, \overline{\varepsilon}, t
					\textup{ chosen such that }\varepsilon\ell^t \textup{ is not a square}\\
	               B^t_{ab}(\ell^{s'}) := &\left\{\begin{pmatrix} x & y \\ 0 & x + \ell^t y\end{pmatrix} : 
	               x \in (\Z/\ell^{s'}\Z)^{\times}, y\in \Z/\ell^{s'}\Z\right\}, \quad 1\leq t \leq {s'}.
	           \end{align*}
	       \end{cor}    
	       \begin{proof}
	           By Proposition~\ref{prop:IndexAbSubgroups}, $H \bmod \Mat_2(\ell^{s'}\Z/\ell^s\Z) \subset \langle I, A' \bmod \ell^{s'}\Z\rangle$ for some $A'$ such that $A'\bmod \ell$ is not a scalar matrix.  If we show that such an $A'$ is conjugate to a matrix of the form
	           \[
	               \begin{pmatrix} x & 0 \\0 & y \end{pmatrix},  \quad
	               \begin{pmatrix} x & \varepsilon \ell^t y \\y & x \end{pmatrix},
	                   \quad\textup{or}\quad
	                   \begin{pmatrix} x & y \\ 0 & x + \ell^t y\end{pmatrix}
	           \]
	           (with $\varepsilon, t$ as above), then $H$ must be contained in $C_s(\ell^{s'}), C^{t, \overline{\varepsilon}}_{ns}(\ell^{s'})$ or $B^t_{ab}(\ell^{s'})$ as desired.
        
	           Since $\overline{A'} := A' \bmod \ell$ is not a scalar matrix, the minimal polynomial of $\overline{A'}$ has degree $2$ and thus is equal to the characteristic polynomial of $\overline{A'}$ which is also equal to the mod $\ell$ reduction of the characteristic polynomial of $A'$.  Thus by~\cite[Theorems II.4 and III.2]{McDonald}, $A'$ is conjugate to a matrix of the form $\begin{pmatrix} x &0\\0&y\end{pmatrix}$ or $\begin{pmatrix}0 & -\det(A')\\1 & \tr(A')\end{pmatrix}$, with the first case occuring if and only if the characteristic polynomial of $A'$ has two roots in $\Z/\ell^s\Z$ that are distinct modulo $\ell$.
        
           Now assume that the characteristic polynomial of $A'$ is reducible and the characteristic polynomial of $A' \bmod \ell$ is a square. Then the roots of the characteristic polynomial are of the form $x, x+\ell^t y$ for some $1\leq t \leq s$, $x,y\in \Z/\ell^s\Z$;  we may assume that $y$ is invertible by taking $t$ maximally. By the above, this is equivalent to the case where $A'$ is conjugate to $\begin{pmatrix} 0 & -x^2 - \ell^txy\\1 & 2x + \ell^ty\end{pmatrix}$.  Then we observe that
   		\[
   			\begin{pmatrix}
   				0 & y\\
   				1 & x + \ell^t y
   			\end{pmatrix}
   			\begin{pmatrix}
   				0 & -x^2 - xy\ell^t\\
   				1 & 2x + \ell^ty
   			\end{pmatrix}
   			\begin{pmatrix}
   				0 & y\\
   				1 & x + \ell^t y
   			\end{pmatrix}^{-1}  =
   			\begin{pmatrix}
   				x & y \\
   				0 & x + \ell^t y
   			\end{pmatrix},
   		\]
   		so $A'$ has the desired form.
		
	           It remains to consider the case when the characteristic polynomial of $A'$ is irreducible in $\Z/\ell^s\Z[T]$.  In this case, since $\ell$ is odd, the discriminant  $\tr(A')^2 - 4\det(A')$ is not a square. Hence, we must have
               \[
                   \det(A') = x^2 - \varepsilon \ell^t y^2,
               \]
               where $0\leq t \leq s - 1$, $\varepsilon$ any fixed lift of an element in $(\Z/\ell^{s-t}\Z)^{\times}\big /(\Z/\ell^{s-t}\Z)^{\times2}$, $y\in (\Z/\ell^s\Z)^{\times}$ and $x= \tr(A')/2$.  Observe that for any $c,d\in \Z/\ell^s\Z$, we have
	           \[
	               \begin{pmatrix}0 & \varepsilon \ell^t y^2 - x^2\\ 1 & 2x \end{pmatrix}
	               \begin{pmatrix}dy - cx & \varepsilon \ell^tcy-dx\\ c & d\end{pmatrix} = 
                   \begin{pmatrix}dy - cx & \varepsilon \ell^tcy-dx\\ c & d\end{pmatrix}
	               \begin{pmatrix} x & \varepsilon \ell^t y \\y & x \end{pmatrix}.
	           \]
	           Since $y$ is a unit, there exists a choice of $c,d\in \Z/\ell^s\Z$ such that $(dy - cx)d - c(\varepsilon \ell^tcy - dx) = (d^2 - \varepsilon \ell^tc^2)y \notin \ell\Z/\ell^s\Z.$  Hence, $A'$ is conjugate to a matrix of the form $\begin{pmatrix} x & \varepsilon\ell^t y \\y & x \end{pmatrix}$.
	       \end{proof}
    
    \subsection{Modular curves}\label{sec:ModularCurves}
   
	      Fix a positive integer $n$.  Let $Y(n)$ be the coarse space of the moduli stack parametrizing elliptic curves with full level-$n$ structure up to isomorphism, i.e., pairs $(E,\psi)$, where $E \to S$ is an elliptic curve over a base scheme $S/\Q$ and $\psi$ is an isomorphism
	       \[
	       \psi\colon (\Z/n\Z \times \Z/n\Z)_S \xrightarrow{\sim} E_{n}.
	       \]
	   We do not require this isomorphism to be compatible with the Weil pairing.  Two pairs $(E,\psi)$ and $(E',\psi')$ are isomorphic if there is an isomorphism $E \to E'$ over $S$ such that the diagram
	   	\[
	   		\xymatrix{
	   			(\Z/n\Z \times \Z/n\Z)_S \ar[r]^-{\psi} \ar[d]_= &  E_{n} \ar[d] \\
	   			(\Z/n\Z \times \Z/n\Z)_S \ar[r]^-{\psi'}  & E'_{n}
	   		}
	   	\]
	   commutes; we write $[E,\psi]$ for the isomorphism class of $(E,\psi)$. The modular curve $Y(n)$ and its compactification $X(n)$, obtained by suitably adding cusps to $Y(n)$, have models defined over $\Q$. These models are not geometrically connected because we did not insist on the compatibility of $\psi$ with the Weil pairing.  In fact, over $\Qbar$, the curve $Y(n)_\Qbar$ has $\varphi(n)$ components (here $\varphi$ denotes the usual Euler totient function), one for each symplectic pairing on $\Z/n\Z\times \Z/n\Z$.  We use $j(E)$ to denote the class of a point $[E]$ in $X(1)$.

	   A subgroup $H \subset \GL_2(\Z/n\Z)$ acts on $Y(n)$ by
	   	\[
	   		h\left([E,\psi]\right) := [E,\psi\circ h] \textup{ for } h \in H,
	   	\]
	   and this action extends uniquely to $X(n)$. The quotient $X(n)/H$ is a curve defined over $\Q$; conjugate subgroups give rise to isomorphic quotient curves~\cite[IV.3]{DeligneRapoport}.  We write $\overline{[E,\psi]}$ for the image of a point $[E,\psi]$ in $X(n)/H$. Points of $X(n)/H$ are related to the Galois representation $\rho_{E,n}$ as follows.

	   \begin{lemma}\label{lem:modcurvestorepns}
		   Let $n$ be a positive integer and let $k$ be a field of characteristic $0$.  There exists a noncuspidal $k$ point $x\in X(n)/H$ if and only if there exists an elliptic curve $E$ over $k$ with $j(E) = j(x)$ such that $\im \rho_{E,n}$ is contained in $H$ (up to conjugacy).
	   \end{lemma}

	   \begin{proof}
		   By~\cite[VI, Prop.~3.2]{DeligneRapoport}, there is a non-cuspidal $k$-point $x\in X(n)/H$ if and only if there exists an elliptic curve $E$ over $k$ with $j(E) = j(x)$ and an isomorphism $\psi\colon (\Z/n\Z)^2 \to E_n$ such that $\overline{[E,\psi]} = x \in (X(n)/H)(k)$.  Then we use the following well-known equivalence: Given an elliptic curve $E$ over $k$, there exists an isomorphism $\psi\colon \Z/n\Z\times \Z/n\Z \to E_n$ such that $\overline{[E,\psi]}$ defines a $k$-point on $X(n)/H$ if and only if $\im \rho_{E,n}$ is contained in a subgroup of $\GL_2(\Z/n\Z)$ conjugate to $H$.
		   In the forward direction, $\psi$ is used to identify $\Aut(E_n)$ with $\GL_2(\Z/n\Z)$ in the definition of the representation $\rho_{E,n}$. In the backward direction, $\rho_{E,n}$ is assumed to come with an identification of $\Aut(E_n)$ with $\GL_2(\Z/n\Z)$, and the representation is used to construct $\psi$. See, e.g.,~\cite[Lemma~2.1]{RouseZureickBrown} for a proof.
	   \end{proof}

	   Fix a prime power $\ell^s$.  We define $X_s(\ell^s)$, $X_{ns}^{t, \overline{\varepsilon}}(\ell^s)$, and $X_{B}^t(\ell^s)$ as the quotient $X(\ell^s)/H$ by taking $H$ to be $C_s(\ell^s)$, $C_{ns}^{t,\epsbar}(\ell^s)$, and $B_{ab}^t(\ell^s)$, respectively, in the notation of~\S\ref{ss:conjugacy}. These curves have natural forgetful maps
	       \[
	       	\pi_{s}^{\ell^s}\colon X_{s}(\ell^s) \to X(1),\quad \pi_{ns}^{\ell^s}\colon \Xns(\ell^s) \to X(1),\quad \textup{and}\quad \pi_{B,t}^{\ell^s}\colon \XB(\ell^s) \to X(1)
	       \]
	   to the modular curve $X(1)$. 
	       
	       \begin{proposition}\label{prop:images}
	       	Let $E$ be an elliptic curve over a field $k$ of characteristic $0$, let $\ell$ be an odd prime, let $s$ be a positive integer, and let $s' := \lceil s/2\rceil$.  If the extension $k(E_{\ell^s})/k$ is abelian, then $j(E) \in X(1)(k)$ is the image of a $k$-rational point from one of the curves $X_{s}(\ell^{s'})$, $\Xns(\ell^{s'})$, or $\XB(\ell^{s'})$.

	   	Conversely, if $j \in Y(1)(k)\subset X(1)(k)$ is the image of a $k$-rational point from one of the curves $X_{s}(\ell^s)$, $\Xns(\ell^s)$,  or $\XB(\ell^s)$, then there exists an elliptic curve $E/k$ with $j(E) = j$, such that the field extension $k(E_{\ell^s})/k$ is abelian.
	       \end{proposition}
    
	       \begin{proof}
             The representation $\rho_{E,\ell^s}$ factors through $\Gal(k(E_{\ell^s})/k)$, so if $k(E_{\ell^s})/k$ is an abelian extension, $\im \rho_{E,\ell^s}$ is an abelian subgroup of $\GL_2(\Z/\ell^s\Z)$. By Corollary~\ref{cor:Conjugacy}, $\im \rho_{E,\ell^{s'}}$ is conjugate to a subgroup of $C_{s}(\ell^{s'})$, $C_{ns}^{t,\epsbar}(\ell^{s'})$, or $B_{ab}^t(\ell^{s'})$. Hence, by Lemma~\ref{lem:modcurvestorepns}, there is a $k$-rational point on one of the curves $X_{s}(\ell^{s'})$, $\Xns(\ell^{s'})$, or $\XB(\ell^{s'})$ whose image in $X(1)$ is $j(E)$.  Conversely, given a $k$-rational point $P$ on one of the curves $X_{s}(\ell^s)$, $\Xns(\ell^s)$, or $\XB(\ell^s)$, by Lemma~\ref{lem:modcurvestorepns} there is an elliptic curve $E/k$ whose $j$-invariant is the image of $P$ in $X(1)$, and such that $\im \rho_{E,\ell^s}$ is conjugate to a subgroup of $C_{s}(\ell^{s})$, $C_{ns}^{t,\epsbar}(\ell^{s})$, or $B_{ab}^t(\ell^{s})$. Therefore $k(E_{\ell^s})/k$ is abelian.
	       \end{proof}

\section{Equivalence of uniform bounds}\label{sec:UniformBounds}

    In this section we prove the equivalence among strong uniform boundedness statements for the size of the Brauer group (modulo algebraic classes) of certain principal homogeneous spaces of abelian surfaces, for their associated K3 surfaces, for the existence of abelian $n$-division fields, and for rational points on certain modular curves.  In addition to the results stated in~\S\ref{s:intro}, we also prove the following.
    
    \begin{thm}\label{thm:EquivalentUniformBoundsNumberField}
    Let $d$ be a positive integer.  The following uniform boundedness statements are equivalent.
    \begin{enumerate}
    	\item[\textbf{(K3)}] For all positive integers $r$, there exists a $B = B(r,d)$ such that for all number fields $k$ of degree at most $r$ and all surfaces $X/k \in\mathscr{K}_d$, 
    	\[
    		\#\left(\frac{\Br X}{\Br_0 X}\right) \leq B.
    	\]
		
    	\item[\textbf{(Ab)}] For all positive integers $r'$, there exists a $B' = B'(r',d)$ such that for all number fields $k'$ of degree at most $r'$ and all $Y/k'\in\mathscr{A}^2_d$, 
    	\[
    		\#\left(\frac{\Br Y}{\Br_1 Y}\right) \leq B'.
    	\]

    	\item[\textbf{(EC)}] For all positive integers $r''$, there exists a $B'' = B''(r'', d)$ such that for all number fields $k''$ of degree at most $r''$ and all non-CM elliptic curves $E/k''$ with a $k''$-rational cyclic subgroup of order $d$,
    	\[
    		\Gal(k''(E_n)/k'') \textup{ is abelian} \Rightarrow n \leq B''.
    	\]
		
    	\item[\textbf{(M)}] For all positive integers $r'''$, there exists a $B''' = B'''(r''', d)$ such that for all number fields $k'''$ of degree at most $r'''$ and all odd prime powers $\ell^s > B'''$, the curves $X_s(\ell^s)\times_{X(1)}X_0(d)$, $\Xns(\ell^s)\times_{X(1)}X_0(d)$, and $\XB(\ell^s)\times_{X(1)}X_0(d)$ have no non-CM, noncuspidal $k'''$-rational points.
    \end{enumerate}
    \end{thm}

    \begin{remark}
    	Let $r$ be a positive integer.  If the modular curve $X_0(d)$ has no non-CM noncuspidal degree $r$ points for $d$ sufficiently large, then Theorem~\ref{thm:EquivalentUniformBoundsNumberField} holds without the dependence on $d$.  Indeed if $X_0(d)$ has no non-CM noncuspidal degree $r$ points for $d$ sufficiently large, then \textbf{(K3)}, \textbf{(Ab)}, \textbf{(EC)}, and \textbf{(M)} trivially hold for $d$ sufficiently large.
    \end{remark}
    
        
    In~\S\ref{sec:BrauerReduction}, we show that the existence of a uniform bound on $\Br X/\Br_0 X$ for K3 surfaces with fixed geometric N\'eron-Severi lattice reduces to a bound on the exponent of $\Br X/\Br_1 X$.  In~\S\ref{sec:PreciseEquivalenceAb}, we show that for $Y\in\mathscr{A}_d^N$, the existence of order $n$ elements in $\Br Y/\Br_1 Y$ is related to the existence of an abelian representation of the Galois group attached to $E_n$, for $E$ some elliptic curve.  These results are used in~\S\ref{sec:EquivalentUniformBoundsArbField} to prove Theorem~\ref{thm:EquivalentUniformBoundsArbField}.  We prove Theorem~\ref{thm:ellprimaryAb} in~\S\ref{sec:UnconditionalBound}, and use it in the proof of Theorem~\ref{thm:EquivalentUniformBoundsNumberField} in~\S\ref{sec:EquivalentUniformBoundsNumberField}. Finally, we prove Theorem~\ref{thm:overQ} in~\S\ref{sec:overQ}.

    \subsection{Uniform bounds on $\Br_1 X/\Br_0 X$}
	\label{sec:BrauerReduction}\label{sec:AlgebraicBrauer}
        	\begin{prop}\label{prop:BrauerReduction}
                Let $X$ be a K3 surface over a field $k$ of characteristic $0$.  Then there exists a positive integer $M$ that depends only on the exponent of $\Br X/\Br_1 X$ such that
                \[
                    \#\left(\frac{\Br X}{\Br_0 X}\right) \leq M.
                \]
            \end{prop}
            \begin{proof}
                Recall that for any smooth variety $W$ over an algebraically closed field of characteristic $0$, we have the following exact sequence
                \[
                    0 \to (\Q/\Z)^{b_2 - r} \to \Br W \to \bigoplus_\ell\HH_{\et}^3(W, \Z_\ell(1))_{\tors} \to 0,
                \]
                where $b_2$ denotes the second Betti number and $r$ denotes the Picard rank.  Hence, for $X$ as in the proposition, we have $\Br \Xbar \isom (\Q/\Z)^{22 - r}$. 
                
                By definition of $\Br_1 X$, we have an inclusion $\Br X/\Br_1 X \hookrightarrow \Br \Xbar \isom (\Q/\Z)^{22 - r}$.  Thus
                \[
                    \#\left(\frac{\Br X}{\Br_1 X}\right) \leq \left(e\left(\frac{\Br X}{\Br_1 X}\right)\right)^{22 - r}
                    \leq \left(e\left(\frac{\Br X}{\Br_1 X}\right)\right)^{21}.
                \]
                The cardinality of $\Br X/\Br_0 X$ is the product of  $\#(\Br X/\Br_1 X)$ and $\#(\Br_1 X/\Br_0 X)$, so it suffices to bound $\#(\Br_1 X/\Br_0 X)$.  Since the geometric Picard group of a K3 surface over a characteristic zero field is free of rank at most $20$, the following lemma completes the proof.
            \end{proof}
            \begin{lemma}\label{lem:AlgBrBound}
        		Let $X$ be a smooth proper geometrically integral variety over a field $k$ of characteristic $0$. Assume that the geometric Picard group $\Pic\Xbar$ is free of rank $r$.  Then there exists a positive integer $M = M(r)$, independent of $X$ such that    $\#\left(\Br_1 X/\Br_0 X\right) \mid M$.
        	\end{lemma}
	
        	\begin{proof}
	   		 The Hochschild-Serre spectral sequence~\eqref{eq:HS} yields an injection
	   		 \[
	   		 	\frac{\Br_1 X}{\Br_0 X}\hookrightarrow \HH^1(\Gamma_k, \Pic \Xbar);
	   		 \]
			 thus it suffices to prove the existence of a constant $M = M(r)$ such that $\#\HH^1(\Gamma_k, \Pic \Xbar) \mid M.$
				
        		Each of the finitely many generators of $\Pic \Xbar$ involves finitely many equations with finitely many coefficients, which generate a finite field extension $K/k$ such that $\Pic X_K \isom \Pic \Xbar$. We may assume that $K/k$ is Galois. As $\Pic \Xbar$ is free and $\Gamma_K$ is profinite, the cohomology group $\HH^1(\Gamma_K,\Pic \Xbar) = \Hom(\Gamma_K,\Pic \Xbar)$ vanishes.  Thus, the inflation--restriction exact sequence from group cohomology gives an isomorphism
        		\[
        			 \HH^1\left(\Gal(K/k),(\Pic \Xbar)^{\Gamma_K}\right)\xrightarrow{\sim} \HH^1(\Gamma_k,\Pic \Xbar).
        		\]
		The injection $\Pic X_K \hookrightarrow (\Pic \Xbar)^{\Gamma_K}$, which comes from the exact sequence of low-degree terms of the Hochschild-Serre spectral sequence~\eqref{eq:HS}, together with the isomorphism $\Pic X_K \isom \Pic \Xbar$ shows that $(\Pic \Xbar)^{\Gamma_K} \isom \Pic \Xbar \isom \Z^r$.  All together, these facts imply that the cohomology group $\HH^1(\Gamma_k,\Pic \Xbar)$ has the form $\HH^1(G, \Z^r)$ for some finite group $G$.  Since $\HH^1(G, \Z^r)$ is $|G|$-torsion, the long exact sequence in cohomology associated to the multiplication-by-$|G|$ map yields an isomorphism
        \[
            \HH^1(G, \Z^r) \isom \frac{\left(\Z^r/|G|\right)^G}{(\Z^r)^G/(|G|)}.
        \]
        Therefore, $\#\HH^1(G, \Z^r)$ divides $|G|^r$, regardless of the action of $G$.
        
         The action of $G$ on $\Z^r$ factors through a finite subgroup $G'$ of $\GL_r(\Z)$, thus $\HH^1(G, \Z^r) = \HH^1(G', \Z^r)$.   Since every finite subgroup of $\GL_r(\Z)$ is a subgroup of $\GL_r(\F_3)$, $|G'|$, and hence $\#\HH^1(G', \Z^r)$, divides a constant that depends only on $r$.  
        	\end{proof}

    \subsection{Order $n$ Brauer classes and abelian representations associated
	 to $E_n$}\label{sec:PreciseEquivalenceAb}

	Let $Y/k\in \mathscr{A}_d^N$ for some $d$ and $N$, and let $A = \Alb(Y)$.  Since $\NS \Ybar \isom \NS \Abar$, by Proposition~\ref{prop:GeomProduct}, there exists a field extension $L/k$ of degree at most $12$, a pair of elliptic curves $E$ and $E'$ over $L$ and $\delta \in L^\times/L^{\times 2}$ such that
	\begin{enumerate}
		\item $A_L\isom E\times E'$, and
		\item there is a cyclic $L$-isogeny of degree $d$ between $E$ and $E'^{\delta}$.
	\end{enumerate}

	\begin{thm}\label{thm:BrauerAbToRep}
		Let $n$ be a positive integer and let $Y/k\in \mathscr{A}_d^N$ for some $d$ and $N$.  Let $L, E, E',$ and $\delta$ be as above.  Suppose that the quotient $\Br Y/\Br_1 Y$ contains an element of order $n$.  Then there is a field extension $\tilde k/k$ of degree at most $(\gcd(\per(Y), n^{\infty}))^4$ such that $\Gal(\tilde L(\sqrt{\delta},E'_{n/c})/\tilde L(\sqrt{\delta}))$ is abelian, where $\tilde L$ is the compositum of $L$ and $\tilde k$ and 
	\begin{equation}\label{eq:DefnOfc}
		c = 
		\begin{cases}
			\gcd(d, n)		& \textup{if } \delta \in {\tilde L}^{\times 2}, \\
			 \gcd(2d^2,n)	& \textup{if } \delta \notin {\tilde L}^{\times 2}.
		\end{cases}
	\end{equation}
	If $n$ is relatively prime to $\per(Y)$, then we may take $\tilde k = k$.
\end{thm}
\begin{remark}
	Theorem~\ref{thm:BrauerAbToRep} also holds with $E'$ replaced by $E$.  Note, however, that if $n/c$ and $d$ are relatively prime then the two statements are equivalent, because then $E'_{n/c}$ and $E_{n/c}$ are isomorphic as $\Gamma_{L(\sqrt{\delta})}$-modules.
\end{remark}

\begin{cor}\label{cor:BrauerKummerToRep}
	Let $X = \Kum Y$ be a Kummer surface over a field $k$ of characteristic $0$ such that $\NS \Xbar \isom \Lambda_d$, and let $n$ be a positive integer.  Since $Y\in \mathscr{A}^2_d$ we have $L, E, E',$ and $\delta$ as above.  
	
	Suppose that the quotient $\Br X/\Br_1 X$ contains an element of order $n$.  Then there is a field extension $\tilde k/k$ of degree at most $16$ such that $\Gal(\tilde L(\sqrt{\delta},E'_{n/c})/\tilde L(\sqrt{\delta}))$ is abelian, where $\tilde L$ is the compositum of $L$ and $\tilde k$ and $c$ is as in~\eqref{eq:DefnOfc}.
	If $n$ is odd or if $X$ is a Kummer surface (i.e., $Y$ is the trivial $2$-covering), then we may take $\tilde k = k$.
\end{cor}
\begin{proof}
	This follows from~\eqref{eq:KummerToPHS}.
\end{proof}

\begin{cor}\label{cor:BrauerK3ToRep}
    Let $X$ be a K3 surface over a field $k$ of characteristic $0$ such that $\NS \Xbar \isom \Lambda_d$, and let $n$ be a positive integer. Suppose that the quotient $\Br X/\Br_1 X$ contains an element of order $n$.  Then there is a field extension $\tilde L/k$ whose degree is uniformly bounded (independent of $X, k, d$ and $n$) and an elliptic curve $E'$ over $\tilde L$ such that $\Gal(\tilde L(E'_{n/\gcd(2d^2,n)})/\tilde L)$ is abelian.
\end{cor}
\begin{proof}
    By Corollary~\ref{cor:Classification}, there exist: a field $L_0/k$ whose degree is uniformly bounded, elliptic curves $E,E'$ over $L_0$, a $2$-covering $Y \to E\times E'$, such that $X_{L_0} = \Kum Y$ and there is a cyclic isogeny $\phi\colon E \to E'$ of degree $d$.  Hence, Corollary~\ref{cor:BrauerKummerToRep} implies that there is an extension  $\tilde k/L_0$ of degree at most $16$ such that $\Gal(\tilde k( E'_{n/c})/\tilde k)$ is abelian for $c$ as in~\eqref{eq:DefnOfc} (since $Y$ is a $2$-covering of $E\times E'$, the extension $L/L_0$ associated to the  Kummer $X_{L_0}$ is the trivial extension).  We have
        $[\tilde k : k] = [\tilde k: L_0] \cdot [L_0:k]
        \leq  16 \cdot [L_0:k].$
    As $[L_0:k]$ is uniformly bounded, so is $[\tilde L:k]$.  To complete the proof, we note that if $\Gal(\tilde L(E'_{n'})/\tilde L)$ is abelian, so is $\Gal(\tilde L(E'_{n''})/\tilde L)$ for any divisor $n''$ of $n'$.
\end{proof}

\begin{proof}[Proof of Theorem~\ref{thm:BrauerAbToRep}]
	Suppose that $\Br Y/\Br_1 Y$ contains an element of order $n$.  Then there exists an integer $n'$ that is a power of $n$ such that $(\Br Y)_{n'}/(\Br_1 Y)_{n'}$ has an element of order $n$.  Hence by Theorem~\ref{thm:BrAbHom}, there exists a field extension $k'/k$ of degree at most $(\gcd(\per(Y), n^{\infty}))^4$ such that
	\[
		\frac{\Hom_{L'}(E_{n'}, E_{n'}')}{\left(\Hom(\Ebar, \Ebar')/{n'}\right)^{\Gamma_{L'}}}
	\]
	contains an element of order $n$; here $L'$ denotes the compositum of $L$ and $k'$.  Furthermore, if $n$ is relatively prime to $\per(Y)$, we may take $k'= k$.  
	
	If $\delta\in {L'}^{\times2}$, then Theorem~\ref{thm:MainHomToEnd} implies that 	\[
		\frac{\End_{L'}(E_{n'}')}{\left(\End(\Ebar')/{n'}\right)^{\Gamma_{L'}}}
	\]
	has an element of order $n/\gcd(d, n)$.  Hence, the result follows from Corollary~\ref{cor:End}.
	
	If $\delta\not\in {`L'}^{\times2}$, then Theorem~\ref{thm:MainHomToEnd} yields an element of order $n/\gcd(n, \gcd(2,n)\gcd(d, n)^2) = n/\gcd(2d^2, n)$ in 
	\[
		\frac{\End_{L'(\sqrt{\delta})}(E'^{\delta}_{n'})}
						{\End_{L'}(E'^{\delta})}.
	\]
	Again, the result follows from Corollary~\ref{cor:End}.
\end{proof}

\begin{thm}\label{thm:RepToBrauerAb}
	Let $n$ be a positive integer, let $E'$ be a non-CM elliptic curve over a field $k$ of characteristic $0$, with a $k$-rational cyclic subgroup $C$ of degree $d$, and  let $E = E'/C$.  Let $W$ and $W'$ be principal homogeneous spaces of $E$ and $E'$, respectively, with periods coprime to $n$, and let $Y = W\times W'$.  (Recall that then $Y\in \mathscr{A}_d^N$ for some integer $N$ coprime to $n$.)  If $\Gal(k(E'_n)/k)$ is abelian, then $\Br Y/\Br_1 Y$ has an element of order $n/\gcd(d, n)$.  
\end{thm}\begin{proof}
	Since $\Gal(k(E'_n)/k)$ is abelian, Corollary~\ref{cor:End} implies that  $\End_k(E'_n)/\left(\End(\Ebar')/n\right)^{\Gamma_k}$ has an element of order $n$.  Hence,  $\Hom_k(E_n, E_n')/\left(\Hom(\Ebar, \Ebar')/n)^{\Gamma_k}\right)$ has an element of order $n/\gcd(d, n)$ by Proposition~\ref{prop:HomToEndRatl}.  Since $n$ is coprime to $\per(W)\per(W')$ and $L = k$, Theorem~\ref{thm:BrAbHom} completes the proof.
\end{proof}
\begin{cor}\label{cor:RepToBrauerK3}
	Let $n, E, E', W, W'$ be as in Theorem~\ref{thm:RepToBrauerAb}.  Suppose further that $W$ and $W'$ have period dividing $2$, so that we may define $X := \Kum (W\times W')$.  If $\Gal(k(E'_n)/k)$ is abelian, then $\Br X/\Br_1 X$ has an element of order $n_{\textup{odd}}/\gcd(d, n_{\textup{odd}})$.  
\end{cor}
\begin{proof}
	This follows from~\cite[Thm. 2.4]{SZ-Kummer} (see~\eqref{eq:KummerToPHS}).
\end{proof}

    \subsection{Proof of Theorem~\ref{thm:EquivalentUniformBoundsArbField}}
    \label{sec:EquivalentUniformBoundsArbField}
	
    	\textbf{(Ab) $\Leftrightarrow$ (K3):} By Lemma~\ref{lem:AlgBrBound}, \textbf{(K3)} is equivalent to {the existence of} a uniform bound on $\#(\Br X/\Br_1 X)$. Furthermore, by  Proposition~\ref{prop:GeometricallyKummerToKummer}, for any $X/k\in \mathscr{K}_d$, there exists a field extension $L$ whose degree is absolutely bounded such that $X_L$ is a  Kummer surface.  Thus, the implications follow from the proof of~\cite[Thm. 2.4]{SZ-Kummer} (see~\eqref{eq:KummerToPHS}).
    
	\smallskip
	
		\textbf{(EC) $\Rightarrow$ (Ab):} Let $r'$ be a positive integer, let $k'/F$ be an extension of degree at most $r$, and let $Y/k' \in\mathscr{A}^2_d$.  Recall that for any abelian surface $A$ over an algebraically closed field with $\rank\NS A = 3$, $\Br A\isom(\Q/\Z)^3$.   As $\Br Y/\Br_1 Y$ injects into $\Br \Ybar$ and $\Ybar$ is an abelian surface, bounding $\#\left(\Br Y/\Br_1 Y\right)$ is equivalent to bounding the exponent of $\Br Y/\Br_1 Y$.  
	
		Assume that there exists an element of order $n$ in $\Br Y/\Br_1Y$, for some odd integer $n$.  Then by Theorem~\ref{thm:BrauerAbToRep} there exists a field extension $k'' := \tilde L(\sqrt{\delta})/k'$ whose degree is bounded by a constant $C$ that is independent of $Y, k', d$ and $n$ and an elliptic curve $E'/k''$ such that $E'$ contains a cyclic subgroup of order $d$ and $k''(E'_{n/\gcd(d^2, n)})/k''$ is an abelian extension.  Since $[k'':F] \leq C\cdot[k':F] \leq Cr'$,  by \textbf{(EC)}, 
	\[
		n \leq \gcd(d^2,n)\cdot B''(Cr', d)\leq d^2B''(Cr', d).
	\]

		\textbf{(Ab) $\Rightarrow$ (EC):} Let $r''$ be a positive integer, let $k''/F$ be an extension of degree at most $r''$, let $E/k''$ be a non-CM elliptic curve with a $k''$-rational cyclic subgroup $C$ of order $d$, and let $E' := E/C$.  Let $n$ be a positive integer such that $k''(E'_n)/k''$ is abelian.  Then by Theorem~\ref{thm:RepToBrauerAb}, there exists $Y/k'' \in \mathscr{A}^2_d$  with an element of order $n/\gcd(d, n)$ in $\Br Y/\Br_1 Y$.   Thus, by assumption
	\[
		n_{\textup{odd}} \leq \gcd(d, n)\cdot B'(r'',d) \leq dB'(r'',d),
	\]
	so we may take $B'' := dB'(r'',d).$ 
	\qed

    \subsection{Uniform bounds on $\ell$-primary parts}
    \label{sec:UnconditionalBound}
	\begin{theorem}\label{thm:ellprimaryEC}
		Let $\ell$ be a prime integer and let $r$ be a positive integer.  There exists a positive constant $B = B(\ell, r)$ such that for all number fields $L$ of degree at most $r$ and all elliptic curves $E/L$, the extension $L(E_{\ell^s})/L$ is nonabelian whenever $s > B$.
	\end{theorem}
	\begin{proof}
		By Proposition~\ref{prop:IndexAbSubgroups} and~\cite{Abramovich-Gonality}, for sufficiently large $s$, the gonality of $X(\ell^s)/H$ is greater than $2r$ for all abelian subgroups $H<\GL_2(\Z/\ell^s\Z)$.  Fix such an ${s =: s_0}$. For any $s\in \mathbb{N}$, consider the set
        \[
            J(s) := \bigcup_{\substack{H<\GL_2(\Z/\ell^{s}\Z)\\ H\textup{ abelian}}} \bigcup_{\substack{L/k\\ [L:k]\leq r}}
            \left\{j(x) : x\in (Y(\ell^{s})/H)(L), \; x\textup{ non-CM}
            \right\}
        \]
        Since there are finitely many abelian subgroups $H<\GL_2(\Z/\ell^{s}\Z)$, a result of Frey~\cite{Frey-InfinitelyManyDegreed} implies that $J(s_0)$ is a finite set.

        Assume that there exists $j\in J(s_0)$. For $E$ a non-CM elliptic curve over $L$, the image of $\rho_{E, \ell^{s_0}}$ is abelian if and only if the image of $\rho_{E', {\ell^{s_0}}}$ is abelian for any twist $E'$ of $E$.  Furthermore, for any elliptic curve,  $\im \rho_{E, \ell^\infty}$ is an open subgroup of $\GL_2(\Z_{\ell})$~\cite[p. IV-11]{Serre-Abelianladic}.  Thus, by Lemma~\ref{lem:modcurvestorepns}, there exists a positive integer $s(j)$ such that for all $s'> s(j)$ and all abelian subgroups $H'<\GL_2(\Z/\ell^{s'}\Z)$, there is no  point on $X(\ell^{s'})/H'$ of degree at most $r$ that maps to $j$. 
               By definition of $s(j)$, we have $J(S) = \emptyset$ whenever   $S > \max_{j\in J(s_0)} s(j)$.  Therefore, after increasing $s_0$ we may assume that $J(s_0)=\emptyset$.

        Assume that $J(s_0)=\emptyset$. For $s' \geq {s_0}$ and $H'< \GL_2(\Z/\ell^{s'}\Z)$ we have $\Q$-morphisms
        \[
            X(\ell^{s'})/H' \To X(\ell^{{s_0}})/(H' \bmod \Mat_2(\ell^{s_0}\Z/\ell^{s'}\Z)),
        \]
        so $J(s') = \emptyset$.  Thus by Lemma~\ref{lem:modcurvestorepns}, for any number field $L$ of degree at most $r$ and any elliptic curve $E/L$, the extension $L(E_{\ell^{s}})$ is nonabelian for all $s \geq s_0$.  Hence we obtain a bound for $s$ that depends only on $\ell$ and $r$.
	\end{proof}

	\begin{proof}[{Proof of Theorem~\ref{thm:ellprimaryAb}}]
	Let $\ell$ be a prime, let $k$ be a number field of degree at most $r$, and let $Y/k\in \bigcup_N\mathscr{A}_d^N$. Assume that the period of $Y$ has $\ell$-valuation at most $v$. Since $\Br Y/\Br_1 Y$ injects into $\Br \Ybar \isom (\Q/\Z)^3$, in order to bound $\#(\Br Y/\Br_1 Y)[\ell^\infty]$ by a constant independent of $Y$, it suffices to bound the exponent of $(\Br Y/\Br_1 Y)[\ell^{\infty}]$ by a constant independent of $Y$.  Assume that there exists an element of order $\ell^s$ in $\Br Y/\Br_1 Y$.  Then by Theorem~\ref{thm:BrauerAbToRep}, there exists a field extension $\tilde L/k$ with $[\tilde L:k]$ bounded by $24\cdot\ell^{4v}$ and an elliptic curve $E'/\tilde L$ such that the extension $\tilde L(E'_{\ell^s/\gcd(2d^2, \ell^s)})/\tilde L$ is abelian.    Thus, by Theorem~\ref{thm:ellprimaryEC}, $s - v_{\ell}(2d^2)$ must be less than $B(\ell, 24\cdot r\cdot \ell^{4v})$, so $s$ is bounded depending only on $\ell,v, d$, and $r$.
	\end{proof}
    
    \begin{proof}[Proof of Corollary~\ref{cor:ellprimaryK3}]
        This follows from Theorem~\ref{thm:ellprimaryAb} and~\cite[Thm. 2.4]{SZ-Kummer} (see~\ref{eq:KummerToPHS}).
    \end{proof}

    \subsection{Proof of Theorem~\ref{thm:EquivalentUniformBoundsNumberField}}
    \label{sec:EquivalentUniformBoundsNumberField}
	
		By Theorems~\ref{thm:ellprimaryEC} and~\ref{thm:ellprimaryAb} and Corollary~\ref{cor:ellprimaryK3} for $\ell = 2$,  \textbf{(EC)}, \textbf{(Ab)}, and \textbf{(K3)} from Theorem~\ref{thm:EquivalentUniformBoundsArbField} are equivalent, respectively, to \textbf{(EC)}, \textbf{(Ab)}, and \textbf{(K3)} from Theorem~\ref{thm:EquivalentUniformBoundsNumberField}.  Hence, it suffices to prove that one of these statements are equivalent to \textbf{(M)}.
	
	\textbf{(EC) $\Rightarrow$ (M):} Let $r'''$ be a positive integer, let $k'''/F$ be a field extension of degree at most $r'''$ and consider an odd prime power $\ell^s$.  Suppose that at least one of the curves $X_s(\ell^s)\times_{X(1)}X_0(d)$, $X_{ns}^{t, \overline{\varepsilon}}(\ell^s)\times_{X(1)}X_0(d)$, or $\XB(\ell^s)\times_{X(1)}X_0(d)$ has a non-CM, noncuspidal $k'''$-rational point. Then by Proposition~\ref{prop:images} there exists an elliptic curve $E/k'''$ together with a $k'''$-rational cyclic subgroup of order $d$ such that $\Gal(k'''(E_{\ell^s})/k''')$ is abelian. Applying \textbf{(EC)} with $r'' := r'''$, there exists a $B'' = B''(r', d)$ such that $\ell^s \leq B''$. Take $B''' = B''$. 

	\smallskip
	
	\textbf{(M) $\Rightarrow$ (EC):} {Let $r''$ be a positive integer, let $k''/F$ be an extension of degree at most $r''$, let $E/k''$ be a non-CM elliptic curve with a $k''$-rational cyclic subgroup $C$, and let $E' := E/C$.  Let $n$ be a positive integer such that $k''(E'_n)/k''$ is abelian. Fix an odd prime $\ell$ and set $s := v_{\ell}(n)$ and $s' := \lceil \frac{s}2\rceil$.  By Proposition~\ref{prop:images},  the point $j(E') \in X(1)(k)$ is the image of a $k''$-rational point of one of the curves $X_{s}(\ell^{s'})$, $X_{ns}^{t, \overline{\varepsilon}}(\ell^{s'})$, or $\XB(\ell^{s'})$.  Since $E'$ is smooth and non-CM, the $k''$-rational point must be noncuspidal and non-CM. Since $E'$ also gives rise to a $k''$-point on $X_0(d)$, using \textbf{(M)} with $r'''= r''$, there is a $B'''$ such that $\ell^{s'} \leq B'''$. Since $B'''$ is independent of $\ell$ and $s'$, we have 
    \[
        n_{\textup{odd}} \;= \;\prod_{\ell > 2}\ell^{s}
		\;\leq\; \prod_{\ell > 2}\ell^{2s'}\; \leq\;
		\prod_{2 < \ell \leq {B'''}^2} {B'''}^2\; =:\; B_{\circ}''.
    \]
    Since $B'''$ depends only on $d$ and $r''' = r''$, $B_{\circ}''$ also depends only on $d$ and $r'$.  By Theorem~\ref{thm:ellprimaryEC}, $v_2(n) \leq b=b(2, F, r'')$, so $n \leq 2^b\cdot B''_{\circ}=:B''$.
	\qed
    
    \subsection{Proof of Theorem~\ref{thm:overQ}}\label{sec:overQ}
        Let $E$ be a non-CM elliptic curve over $\Q$ with a $\Q$-rational cyclic subgroup $C$, let $E' := E/C$, let $Y = E'\times E$, and let $X := \Kum Y$.  By~\cite[Thm. 2.4]{SZ-Finiteness}, there is an injective map $\Br X/\Br_1 X\hookrightarrow\Br Y/\Br_1Y$, thus it suffices to only bound $\#(\Br Y/\Br_1 Y)$.
		
		 Assume that $\Br Y/\Br_1 Y$ contains an element of order $n$. Then $\Q(E'_{n/\gcd(d, n)})$ is an abelian extension of $\Q$, by Theorem~\ref{thm:BrauerAbToRep}.  Hence, work of Gonz\'alez-Jim\'enez and Lozano-Robledo~\cite{GJLR-AbelianDivisionFields} implies that $n/\gcd(d, n) \leq 8$ so $n \leq 8d$.  Since $\Br Y/\Br_1 Y$ injects into $\Br \Ybar\isom (\Q/\Z)^3$, we conclude that 
        \[
            \#\left(\frac{\Br Y}{\Br_1 Y}\right) \leq (8d)^3. \tag*{\qed}
        \]	

\appendix
\section{Corrigendum}
{This is not an appendix, at least in a conventional sense.  Rather, it is a corrigendum that will be published separately from the article, as the main body of the article has already appeared in print.  We have added the corrigendum as an appendix in the arxiv posting as a means of alerting a reader to its existence.}\\

There is an error in the statement and proof of Proposition~\ref{prop:ranks} that affects the statements of Corollaries~\ref{cor:End} and~\ref{cor:End2}. 
In this note, we correct the statement of Proposition~\ref{prop:ranks} and explain how to rectify subsequent statements. In brief, for a statement about abelian Galois representations of a \emph{fixed} level, ``abelian'' should be replaced with ``liftable abelian'' (Definition~\ref{defn:liftableabelian}).  Statements about abelian Galois representations of arbitrarily high level, however, remain unchanged because such representations give rise to liftable abelian Galois representations of smaller, but still arbitrarily high, level.  Hence the main theorems of the paper remain unchanged.

\begin{defn}
\label{defn:liftableabelian}
Let $n$ be a positive integer.  A subgroup $H$ of $\GL_2(\Z/n\Z)$ is \defi{liftable abelian} if there exists {an abelian subgroup} $\widehat H < \GL_2(\widehat\Z))$ such that $\widehat H$ surjects onto $H$ under the natural quotient map $\GL_2(\widehat\Z) \onto \GL_2(\Z/n\Z)$. (In particular, a liftable abelian subgroup is abelian.)
\end{defn}

For a positive integer $n$, and an elliptic curve $E$ over a field $k$ of characteristic $0$, let $\rho_{E,n}\colon \Gamma_k \to \Aut(E_{n}) \isom \GL_2(\Z/n\Z)$ denote the representation arising from the action of Galois on the $n$-torsion of $E$. If $m\mid n$, then we write $\iota_{m,n} \colon E_m \into E_n$ for the natural inclusion; the image of the multiplication map 
\begin{equation*}
\left[\frac{n}{m}\right]\colon \End(E_m) \to \End(E_n), \qquad
\varphi \mapsto \iota_{m,n}\circ\varphi\circ\left[\frac{n}{m}\right]
\end{equation*}
is $\End(E_n)\cap M_2(\frac{n}{m}\Z/n\Z)$. This map is also compatible with the homomorphisms $\rho_{E,n}$ and $\rho_{E,m}$. These two observations together yield the following lemma, where we have written $\End_k(E_m)\circ\left[\frac{n}{m}\right]$ in place of $\left[\frac{n}{m}\right]\left(\End_k(E_m)\right)$, to match the notation in the original paper.
\begin{lemma}
\label{lem:clarity}
If $m\mid n$ then $\End_k(E_m)\circ\left[\frac{n}{m}\right] = \End_k(E_n) \cap M_2(\frac{n}{m}\Z/n\Z)$.
\qed
\end{lemma}

Write $G_{E, n}$ for the quotient of $\im \rho_{E,n}$ by the subgroup of scalar matrices. The following proposition and corollaries replace Proposition~\ref{prop:ranks} and Corollaries~\ref{cor:End} and~\ref{cor:End2}.

\begin{prop}\label{prop:Correction}
    Let $\ell$ be a prime, let $s$ be a positive integer, and let $E$ be an elliptic curve over a field $k$ of characteristic $0$.  Then
       \[
            \dim_{\F_{\ell}}
            \frac{\End_{k}(E_{\ell^s})}
            {{\End_{k}(E_{\ell^{s-1}})\circ [\ell]}}
            =
            \begin{cases}
                4& \textup{if } G_{E, \ell^s} = \{1\},\\
                2& \textup{if } G_{E, \ell^s} \neq \{1\} \textup{ and $\im(\rho_{E,\ell^s})$ is liftable abelian, and}\\
                1& \textup{if }\im(\rho_{E,\ell^s})\textup{ is not liftable abelian.}
            \end{cases}
        \]
\end{prop}
\begin{cor}\label{cor:CorrectedEnd}
    Let $E$ be an elliptic curve over $k$ and let $n$ be a positive integer.  Then we have an isomorphism of abelian groups
    \[
        \End_{k}(E_n) \isom 
        \Z/n\Z \times \Z/n_1\Z \times \left(\Z/n_2\Z\right)^2 
    \]
    for positive integers $n_2\mid n_1\mid n$.  Furthermore, $n_1$ is the largest integer dividing $n$ such that $\Gal(k(E_{n_1})/k)$ is liftable abelian and $n_2$ is the largest integer dividing $n$ such that $\Gal(k(E_{n_2})/k) \subset (\Z/n_2\Z)^{\times}$ where $a\in (\Z/n_2\Z)^{\times}$ acts by $P\mapsto aP$.  If $E$ is non-CM, then $(\End(\Ebar)/n)^{\Gamma} \isom \Z/n\Z$ and hence
    \[
        \frac{\End_k(E_n)}{(\End(\Ebar)/n)^{\Gamma}} \isom \Z/n_1\Z \times \left(\Z/n_2\Z\right)^2.
    \]
\end{cor}
\begin{remark}\label{rem:clarity}
	The proofs of Proposition~\ref{prop:Correction} and Corollary~\ref{cor:CorrectedEnd} prove a stronger statement, namely that if $n_1 = n_2$, then $\End_{k}(E_n) = \Z/n\Z\cdot I + \Mat_2({\textstyle \frac{n}{n_2}}\Z/n\Z)$ and if $n_1\neq n_2$, then
	\[
		\End_{k}(E_n) = \left\{aI + b{\textstyle \frac{n}{n_1}}A:a,b\in \Z/n\Z\right\} + \Mat_2({\textstyle \frac{n}{n_2}}\Z/n\Z)\subset \Mat_2(\Z/n\Z),  
	\]
	where $A\in \im \rho_{E, n_1}$ is a matrix such that $A\bmod \ell \notin \langle I\rangle$, for any $\ell\mid n_1$, as given by Lemma~\ref{lem:rank2liftableabelian}.
\end{remark}
\begin{cor}
\label{cor:CorrectedEndQuotient}
    Let $E$ be an elliptic curve over $k$ and let $n$ be a positive integer.  Let $k'/k$ be a field extension. There is an isomorphism of abelian groups
    \[
        \frac{\End_{k'}(E_{n})}{\End_{k}(E_{n})} \isom \Z/{\textstyle \frac{n_1'}{n_1}}\Z \times \left(\Z/{\textstyle \frac{n_2'}{n_2}}\Z\right)^2,
    \]
    where $n_1'$ (respectively $n_1$) is the largest integer dividing $n$ such that $\Gal(k'(E_{n_1})/k')$ (respectively $\Gal(k(E_{n_1})/k)$) is liftable abelian and $n_2'$ (respectively $n_2$) is the largest integer dividing $n$ such that $\Gal(k'(E_{n_2})/k') \subset (\Z/n_2'\Z)^{\times}$ (respectively $\Gal(k(E_{n_2})/k) \subset (\Z/n_2\Z)^{\times}$).
\end{cor}

The proof of Corollary~\ref{cor:CorrectedEnd} proceeds as in the original proof of~\ref{cor:End}, almost verbatim, after replacing the word ``abelian'' with the phrase ``liftable abelian'', and references to~Proposition~\ref{prop:ranks} with references to Proposition~\ref{prop:Correction}. Corollary~\ref{cor:CorrectedEndQuotient} is more easily deduced from Remark~\ref{rem:clarity}.
    
Characterizing liftable abelian groups is essential in the proof of Proposition~\ref{prop:Correction}:
\begin{lemma}\label{lem:rank2liftableabelian}
    Let $n$ be a positive integer and let $H < \GL_2(\Z/n\Z)$ be a subgroup. Then $H$ is liftable abelian if and only if $H$ is contained in a subring of $\Mat_2(\Z/n\Z)$ generated by $I$ and $A$, for some matrix $A$ such that $A \bmod \ell \notin \langle I\rangle$ for any prime $\ell\mid n$. 
\end{lemma}
\begin{proof}
    By the Sun Tzu Remainder Theorem, it suffices to prove the lemma in the case where $n=\ell^s$, so we restrict to this case for the remainder of the proof.

    If $H$ is liftable abelian, then there is an abelian subgroup $\widehat H < \GL_2(\Z_\ell)$ that surjects onto $H$. Applying Proposition~\ref{prop:IndexAbSubgroups} to $\widehat H \bmod \ell^{2s}$, we conclude there is an $A' \in M_2(\Z/\ell^{2s}\Z)$ with $A'\bmod \ell \notin\langle I\rangle$ such that $\widehat H \bmod \ell^s = H \subseteq \langle I, A'  \bmod \ell^s \rangle$. 

    Now assume that $H\subset \langle I, A \rangle$ for some matrix $A$ such that $A \bmod \ell \notin \langle I\rangle$.  
    By the proof of Corollary~\ref{cor:Conjugacy} (starting from the second line, taking $s'=s$), any such $H$ is conjugate to a subgroup of $C_s(\ell^s)$, $C_{ns}^{t, \overline{\varepsilon}}(\ell^{s})$, or $B_{ab}^{t}(\ell^{s})$.  Any subgroup of a liftable abelian subgoup is itself liftable abelian, so it suffices to show that the groups $C_s(\ell^k)$, $C_{ns}^{t, \overline{\varepsilon}}(\ell^{k})$, and $B_{ab}^{t}(\ell^{k})$ appearing in Corollary~\ref{cor:Conjugacy} are liftable abelian.

    For the split Cartan group $C_s(\ell^k)$ and the Borel groups $B_{ab}^{t}(\ell^{k})$, the inverse limits
    \[
    \varprojlim_n C_s(\ell^n) \quad\textrm{and}\quad \varprojlim_n B_{ab}^{t}(\ell^{n}).
    \]
    of $\GL_2(\Z_{\ell})$ are abelian and surject onto $C_s(\ell^k)$ and $B_{ab}^{t}(\ell^k)$, respectively, proving the claim. 

    For the group $C_{ns}^{t, \overline{\varepsilon}}(\ell^{k})$, one can construct a surjective system
    \[
    C_{ns}^{t, \overline{\varepsilon_{k+1}}}(\ell^{k+1}) \onto C_{ns}^{t, \overline{\varepsilon_k}}(\ell^{k}),
    \]
    by taking $\revedit{\overline{\varepsilon_{k+1}}} \in (\Z/\ell^{k+1-t})^\times/(\Z/\ell^{k+1-t})^{\times 2}$ a lift of ${\overline{\varepsilon_{k}}} \in (\Z/\ell^{k-t})^\times/(\Z/\ell^{k-t})^{\times 2}$ compatibly up through the system. The inverse limit of this system demonstrates that $C_{ns}^{t, \overline{\varepsilon}}(\ell^{k})$ is liftable abelian.
\end{proof}

\begin{proof}[Proof of Proposition~\ref{prop:Correction}]
   	If $G_{E,\ell^s}=\{1\}$ then the claim is immediate; thus we may assume that $G_{E,\ell^s}\neq 1$.  Now assume further that $\im \rho_{E,\ell^s}$ is liftable abelian. By Lemma~\ref{lem:rank2liftableabelian}, there is an $A' \in M_2(\Z/\ell^s\Z)$ with $A' \bmod \ell \notin \langle I\rangle$ such that $\im \rho_{E,\ell^s} \subseteq \langle I, A' \rangle$.  
    Let $t$ be the maximal integer such that $\im \rho_{E,\ell^s}\subseteq\langle I, \ell^t A'\rangle$.  Note that since $G_{E,\ell^s}\neq \{1\}$, $t$ must be strictly less than $s$, or equivalently, $s - t\geq 1$. Then
    \begin{align*}
      \End_k(E_{\ell^s}) = & \{M : A'M \equiv MA' \bmod \ell^{s-t}\}  \\
      = & \{ aI + bA' + \ell^{s-t}M' : a,b\in \Z/\ell^s{\Z}, M'\in \Mat_2(\Z/\ell^s\Z)\},
    \end{align*}
    where the second equality comes from Lemma~\ref{lem:CommutingMatrices} applied to $A'$.
	Together with Lemma~\ref{lem:clarity}, we deduce that for $M'\in \Mat_2(\Z/\ell^{s}\Z)$, we have $\ell^{s-t-1}\cdot(\ell M')\in \End_{k}(E_{\ell^{s-1}})\circ [\ell]$.
    Hence, $\End_k(E_{\ell^s}) / \End_{k}(E_{\ell^{s-1}})\circ [\ell]$ is generated by $I$ and $A'$ and so is $2$-dimensional.

    To complete the proof, we claim that if $\dim_{\F_{\ell}}
    \frac{\End_{k}(E_{\ell^s})}
    {{\End_{k}(E_{\ell^{s-1}})\circ[\ell]}}\geq 2$ then $\im \rho_{E, \ell^s}$ is liftable abelian.  The identity endomorphism always generates a one-dimensional subspace of this quotient, so if the inequality holds, then by Lemma~\ref{lem:clarity} there exists an $A\in \End_{k}(E_{\ell^s})$ that is not a scalar modulo $\ell$.  Since every element of $\im \rho_{E,\ell^s}$ commutes with $A$, by Lemma~\ref{lem:CommutingMatrices} we have $\im \rho_{E, \ell^s}\subset \langle I, A \rangle$.  Lemma~\ref{lem:rank2liftableabelian} then shows that $\im \rho_{E, \ell^s}$ is liftable abelian.
\end{proof}

\subsection*{Corrections for subsequent statements in the main paper.}

Corollary~\ref{cor:End} is used in the proof of Theorems~\ref{thm:BrauerAbToRep}and~\ref{thm:RepToBrauerAb}.  In turn, Theorem~\ref{thm:BrauerAbToRep} is used in the proofs of Theorems~\ref{thm:EquivalentUniformBoundsArbField} \textbf{(EC) $\Rightarrow$ (Ab):},~\ref{thm:ellprimaryAb}, and~\ref{thm:overQ}, while Theorem~\ref{thm:RepToBrauerAb} is used in the proof of Theorem~\ref{thm:EquivalentUniformBoundsArbField} \textbf{(Ab) $\Rightarrow$ (EC):}, and Corollary~\ref{cor:RepToBrauerK3}.

Using Corollary~\ref{cor:CorrectedEnd} in place of Corollary~\ref{cor:End} in the proof of Theorem~\ref{thm:BrauerAbToRep} yields a \emph{stronger} version of the theorem.  Namely, Corollary~\ref{cor:CorrectedEnd} allows one to deduce that the Galois group $\Gal(\tilde{L}(\sqrt{\delta}, E'_{n/c})/\tilde{L}(\sqrt{\delta}))$ is \emph{liftable} abelian.  In particular, the proofs of Theorems~\ref{thm:EquivalentUniformBoundsArbField} \textbf{(EC) $\Rightarrow$ (Ab):},~\ref{thm:ellprimaryAb}, and~\ref{thm:overQ} go through unchanged.

On the other hand, using Corollary~\ref{cor:CorrectedEnd} in place of Corollary~\ref{cor:End} in the proof of Theorem~\ref{thm:RepToBrauerAb} yields the following weaker version of the theorem and its corollary; however, this version still suffices for the proof of Theorem~\ref{thm:EquivalentUniformBoundsArbField} \textbf{(Ab) $\Rightarrow$ (EC):}.
\begin{thm}\label{thm:CorrectedRepToBrauerAb}
	Let $n$ be a positive integer, let $E'$ be a non-CM elliptic curve over a field $k$ of characteristic $0$, with a $k$-rational cyclic subgroup $C$ of order $d$, and  let $E = E'/C$.  Let $W$ and $W'$ be principal homogeneous spaces of $E$ and $E'$, respectively, with periods coprime to $n$, and let $Y = W\times W'$. If $\Gal(k(E'_n)/k)$ is liftable abelian, then $\Br Y/\Br_1 Y$ has an element of order $n/\gcd(d, n)$.
\end{thm}
\begin{remark}
    With the weaker assumption that $\Gal(k(E'_n)/k)$ is abelian, then one can prove that $\Br Y/\Br_1 Y$ has an element of order $m/\gcd(d, m)$, where $m := \prod_{p|n} p^{\lceil v_p(n)/2\rceil}$.
\end{remark}
\begin{cor}
Let $n$, $E$, $E'$, $W$, $W'$ be as in Theorem~\ref{thm:CorrectedRepToBrauerAb}. Suppose further that $W$ and $W'$ have period dividing $2$, so that we may define $X:= \Kum(W\times W')$. If $\Gal(k(E'_n)/k)$ is liftable abelian, then $\Br X/\Br_1 X$ has an element of order $n_\textup{odd}/\gcd(d,n_{\textup{odd}})$.
\end{cor}

Finally, Theorem~\ref{thm:RepToBrauerAb} is used in the proof of the implication \textbf{(Ab) $\Rightarrow$ (EC)} in Theorem~\ref{thm:EquivalentUniformBoundsArbField}.

\begin{proof}[Corrected proof of Theorem~\ref{thm:EquivalentUniformBoundsArbField}\textbf{(Ab) $\Rightarrow$ (EC):}]
    Let $r''$ be a positive integer, let $k''/F$ be an extension of degree at most $r''$, let $E/k''$ be an elliptic curve with a $k''$-rational cyclic subgroup $C$ of order $d$, and let $E' := E/C$.  Let $n$ be a positive integer such that $k''(E'_n)/k''$ is abelian. Then Lemma~\ref{lem:rank2liftableabelian} and Proposition~\ref{prop:IndexAbSubgroups} together imply that $\Gal(k(E'_{m})/k)$ is liftable abelian, where $m := \prod_{p|n} p^{\lceil v_p(n)/2\rceil}$.  Hence, Theorem~\ref{thm:CorrectedRepToBrauerAb} implies that there exists $Y/k'' \in \mathscr{A}^2_d$  with an element of order $m/\gcd(d, m)$ in $\Br Y/\Br_1 Y$, where $m = \prod_{p|n}p^{\lceil v_p(n)/2\rceil}$.   Thus, by assumption
		$m_{\textup{odd}} \leq B'(r'',d)\gcd(d, m) \leq B'(r'',d)d.$
    Since $n\mid m^2$, we may take $B'' := (B'(r'',d)d)^2.$ 
\end{proof}

\subsection*{Acknowledgements}
We are grateful to Francesca Balestrieri, Alexis Johnson, and Rachel Newton for alerting us to the error in the proof of~Proposition~\ref{prop:ranks}, and for an insightful counterexample that led us to a corrected statement.  We also thank the referee for a careful reading of the Corrigendum, and for their valuable suggestions.


\end{document}